\documentclass{amsart}
\address{\newline{\normalsize Max Planck Institute for Mathematics, Vivatsgasse 7, 53111 Bonn, Germany}
\newline{\it E-mail address}: karzhema@mpim-bonn.mpg.de}
\usepackage{amscd,amsthm,amsmath,amssymb}
\usepackage[matrix,arrow]{xy}

\makeatletter\@addtoreset{equation}{section}\makeatother

\makeatletter\@addtoreset{subsection}{equation}\makeatother

\newtheorem{theorem}[equation]{Theorem}
\newtheorem{prop}[equation]{Proposition}
\newtheorem{lemma}[equation]{Lemma}
\newtheorem{cor}[equation]{Corollary}
\newtheorem{conj}[equation]{Conjecture}

\theoremstyle{remark}
\newtheorem{remark}[equation]{Remark}

\newtheorem*{inter}{Intermedia}

\theoremstyle{definition}
\newtheorem{definition}[equation]{Definition}
\newtheorem{example}[equation]{Example}
\newtheorem{question}{Question}

\renewcommand{\thequestion}{C}

\newcommand{\com}{\mathbb{C}}
\newcommand{\ra}{\mathbb{Q}}
\newcommand{\f}{\mathbb{F}}
\newcommand{\summ}{\displaystyle\sum}
\newcommand{\aut}{\text{Aut}}
\newcommand{\p}{\mathbb{P}}
\newcommand{\fie}{{\bf k}}

\newcommand{\map}{\longrightarrow}
\newcommand{\cel}{\mathbb{Z}}
\newcommand{\gal}{\mathrm{Gal}}
\newcommand{\bra}{\mathrm{Br}}

\renewcommand{\theenumi}{\arabic{enumi}}

\title{On one stable birational invariant}

\author{Ilya Karzhemanov}
\begin{document}

\maketitle

\begin{abstract}
This is an expository article in which we propose that (rational)
fibrations on the projective space $\p^n$ by (birationally)
Abelian hypersurfaces, for an arbitrary $n\ge 2$, provide an
obstruction to stable rationality of algebraic varieties. We
discuss the evidence for this proposition and derive some (almost
straightforward) corollaries from it.
\end{abstract}

\bigskip

\section{Introduction}
\label{section:int}

\refstepcounter{equation}
\subsection{}
\label{subsection:int-1}

Let $X$ be an algebraic variety defined over a field $\fie$. If
not stated otherwise, $\text{char}~\fie = 0$ and $X$ will be
smooth, projective and geometrically integral. We will denote by
$\fie(X)$ the field of rational functions on $X$. When speaking
about \emph{stable birational geometry} of $X$, one is usually up
to some ((bi)rational) interrelation between $X$ and its
``stabilization" $X\times\p^k$, with an arbitrary $k\ge 1$. More
specifically, one is interested in those properties of the field
$\fie(X)$ that (dis)appear when passing to the field
$\fie(X)(t_1,\ldots,t_k)$, $t_i$ being $\fie(X)$-transcendental
variables. For instance, one may study such classical question as
(stable) rationality of linear quotients $X = V/G$ (for $V :=
\fie^{\dim X}$, $G\subseteq GL(V)$ a reductive group and the
$G$-action being free at the generic point on $V$), and we refer
to \cite{colliot-sansuc} (see also \cite{manin-tsfas},
\cite{prokhor-quots}) and references therein for an extensive
overview of the state of art. In its turn, a particular (and in
fact the only one) really unavoidable matter in this discussion is
the need of \emph{stable birational invariants} of $X$, i.\,e.
those properties of $\fie(X)$ that ``do not change" after passage
to $\fie(X)(t_1,\ldots,t_k)$ and/or vice versa.

Our aim in this introductory note is like this. First we will give
an account (mostly in this section) of some of classical stable
birational invariants of $X$ (see {\ref{subsection:int-2}},
{\ref{subsection:int-3}} and {\ref{subsection:int-5}} below). Next
we will formulate a problem which seems to be not accessible by
the classical tools (see {\ref{subsection:int-4}},
Question~\ref{question:colliot}). Then we introduce a (seemingly
new) stable birational invariant of $X$ (see
{\ref{subsection:int-5}}, {\ref{subsection:int-66}}), conjecture
one of its crucial properties
(Conjecture~\ref{theorem:p-n-ab-pen}), and after that we deduce
Theorem~\ref{theorem:main}. The latter is proved completely in
Section~\ref{section:pro}. We remark that
Theorem~\ref{theorem:main} (as well as
Conjecture~\ref{theorem:p-n-ab-pen}) has many quite strong
implications, among which we mention only the solution to our
initial problem (= Question~\ref{question:colliot}) and another
interesting corollary (= Corollary~\ref{theorem:main-cor}), all
treated in {\ref{subsection:int-6}}. The rest of the paper
(Sections~\ref{section:misc} and \ref{section:misce}) is devoted
to the verification (or, if one likes, to the sketch of a proof)
of Conjecture~\ref{theorem:p-n-ab-pen}.

As a result of the overview style we have adopted, the text
contains an abundance of Remarks and \emph{Intermedias}, aiming to
guide the reader through the line of arguments we have employed to
attack Conjecture~\ref{theorem:p-n-ab-pen}, as well as to draw
some parallels with relevant theories (the latter are heuristics
really and we apologize in advance for a loose exposition). To sum
up, our arguments here mimic in a sense (and were inspired by)
those in \cite{mumford-such-surf}, used to construct an algebraic
surface (over $\ra(\sqrt[7]{1}$)), different from $\p^2$ and
having ample $K,~K^2 = 9,~q=p_g=0$. In fact, this surface is
``homeomorphic" to a degree $2$ del Pezzo surface, when both
brought to $\ra_2$, -- the fact established in
\cite{mumford-such-surf} by means of $p$-adic uniformization (cf.
Intermedia in {\ref{subsection:misc-5}} below).

\refstepcounter{equation}
\subsection{}
\label{subsection:int-2}

To start with, we indicate that the above mentioned classical
invariants are essentially ``group-theoretic". For instance, let
$G$ be a profinite group and $M$ be a free $\cel$-module of finite
rank, with continuous $G$-action ($M$ is called \emph{$G$-module}
for short). Then $M$ admits what is called \emph{flascque
resolvent} (see \cite{voskresenskij}), i.\,e. there is an exact
sequence
$$
0\to M\to D\to F\to 0
$$
of $G$-modules, where $D$ carries a $\cel$-basis in which $G$ acts
by permutations ($D$ is called a \emph{permutation module}), and
$\text{Hom}(H^1(G',F),\ra/\cel) = 0$ for any open subgroup
$G'\subseteq G$ ($F$ is called \emph{flascque module}).

\begin{example}
\label{example:torus-pic} Let $T$ be a $\fie$-torus and $X\supset
T$ a smooth projective completion of $T$. Put $\bar{X} := X
\times_{\fie} \bar{\fie},~\bar{T} := T \times_{\fie} \bar{\fie}$
and consider the group $D$ of divisors supported on
$\bar{X}\setminus\bar{T}$. Let also $\hat{T} :=
\text{Hom}(T,\bar{\fie}^*)$ be the group of characters of $T$.
Then $\text{Pic}(\bar{X})$ is flasque and one gets a flascque
resolvent $0 \to \hat{T} \to D \to \text{Pic}(\bar{X})\to 0$ of
$\hat{T}$.
\end{example}

Further, considering various $M$ up to the direct sums with
permutation $G$-modules, one gets a homomorphism $\rho: S_G \to
F_G$ given by $M\mapsto \rho(M) := F$, where $S_G$ is the
semigroup (w.\,r.\,t. $\oplus$) of classes $\left[M\oplus P\ \vert
\ P\ \text{a permutation $G$-module}\right]$ and $F_G$ is the
semigroup (of classes of) flascque $G$-modules. Now, if say
$L\supseteq\fie$ is a Galois extension with $\text{Gal}(L/\fie) =
G$, then $\rho(X) :=
\rho(H^0(U_L,\mathcal{O}_{U_L}^*)\slash\fie^*)\in F_G$ is a stable
($\fie$-)birational invariant of $X$. Here $U\subset X$ is an open
subset such that $\text{Pic}(U_L) = 0$ for $U_L :=
U\times_{\fie}L$ (see \cite{voskresenskij}). Furthermore, one gets
$\rho(X) = 0$ when $X$ is $\fie$-rational, thus an obstruction to
stable $\fie$-rationality.

\refstepcounter{equation}
\subsection{}
\label{subsection:int-3}

Let us now consider the \emph{Brauer group} $\mathrm{Br}(\fie(X))$
of the field $\fie(X)$. One distinguishes a subgroup
$\bra_{\emph{v}}(\fie(X))\subset\bra(\fie(X))$ (or simply
$\bra_{\emph{v}}$ if $\fie(X)$ is clear from the context) -- the
\emph{unramified Brauer group} of $\fie(X)$ -- by the property
that for every element $\gamma\in\bra_{\emph{v}}$ and any
valuation $v$ of $\fie(X)$, $\gamma\in\bra(A_{v})$ for the
valuation ring $A_v$ of $v$. Now, to be able to work with
(actually calculate) $\bra_v$, one interprets $\fie(X)$ as the
field of invariants in the algebraic closure $\overline{\fie(X)}$
of the group $G := \gal(\overline{\fie(X)}/\fie(X))$. Thus $X$, in
a sense, is a ``geometric quotient $U/G$" (see
\cite{fed-yu-univ-spaces-1}, \cite{fed-yu-univ-spaces-2},
\cite{fed-yu-univ-spaces-3} for the description of the (universal)
space $U$ (be advised though that $\fie = \overline{\f}_p$ in
\cite{fed-yu-univ-spaces-2}, \cite{fed-yu-univ-spaces-3}); we
remark that the current discussion is merely a heuristics aimed to
develop some intuition for $\bra_{\emph{v}}$ rather than rigorous
statements). Then, for $G$ being a profinite group,
$\bra(\fie(X))$ can be interpreted as $\varinjlim
H^2(G_i,\ra/\cel)$ for some \emph{finite} groups $G_i$ (quotients
of $G$). In fact, for each $\gamma\in H^2(G_i,\ra/\cel) =
H^2(G_i,\overline{\fie(X)}^{*})$ one constructs a central
extension $G_{i,\gamma}$ of $G_i$ and then defines a $\p^m$-bundle
(for some $m = m(\gamma)\in\mathbb{N}$) on (a smooth birational
model of) $X$ as the quotient $(\p^m\times
X_i)\slash_{\displaystyle G_{i,\gamma}}$, where
$X_i\slash_{\displaystyle G_{i,\gamma}}\approx
X$\footnote{$\approx$ stands for the birational equivalence.} and
$X_i := U/_{\displaystyle\text{Ker}[G\to G_i]}$ for the previously
mentioned $U$. We will write, informally, $\bra(\fie(X)) =
H^2(G,\overline{\fie(X)}^{*})$. One may now apply the computations
from \cite{fedya-1} and \cite{fedya-2} to identify
$\bra_{\emph{v}}$ with the subgroup of those $\gamma\in
\bra(\fie(X))$ that restrict trivially on all the rank $2$ finite
abelian subgroups in $G$ with cyclic image in the decomposition
group of $v$ (for all $v$). It also follows from the $\p^m$-bundle
description of $\bra(\fie(X))$ that this group (together with
$\bra_{\emph{v}}$) is a stable birational invariant of $X$ (see
e.\,g. \cite{shaf-luroth}).

\begin{remark}
Of course, there are many more (stable) birational invariants (all
for $\bar{\fie}$-rational $X$), such as the sets
$X(\fie)\slash\text{``R-equivalence"},~X(\fie)\slash\text{``Brauer
equivalence"}$ and the groups
$\text{CH}_0(X),~\text{Ker}\big(\deg: \text{CH}_0(X)\to\cel\big)$
of $0$-cycles on $X$, as well as numerous problems on interactions
between these (in the context of the Hasse principle for example).
Let us stop, however, our listing of invariants at this point and
proceed to the main subject of this paper.
\end{remark}

\refstepcounter{equation}
\subsection{}
\label{subsection:int-4}

Recall that variety $X$ is called \emph{stably b-infinitely
transitive} if $X\times\p^k$ admits an infinitely transitive model
(w.\,r.\,t. the group $\text{SAut}$) for some $k$ (see \cite{bkk}
for the basic definitions). It is conjectured (= \cite[Conjecture
1.4]{bkk}) this property is equivalent to $X$ being unirational
(see \cite{arh-et-al} for an implication in one direction). In
order to develop an approach to the conjecture one may consider
the next

\begin{question}[{= \cite[Remark
3.3]{bkk}}] \label{question:colliot} Suppose $X$ is unirational.
Does then $X\times\p^k$, some $k\ge 0$, admit a model which is
infinitely transitive and carries an algebraic group $\frak{G}$
action with a Zariski dense orbit ($\dim \frak{G}<\infty$)?
\end{question}

It is not difficult to work out positive answer to
Question~\ref{question:colliot} in case $X$ is a linear quotient
$V\slash G$, as in {\ref{subsection:int-1}}, with finite $G$ at
least (cf. \cite[3.1]{bkk}). Thus, in view of the discussion
started in {\ref{subsection:int-2}}, it is tempting to test
Question~\ref{question:colliot} in a more geometric set-up, namely
that when $X$ is a hypersurface (see {\ref{subsection:int-6}}).
But then we have to look for such a (stable) birational property
of $X$ that would ``feel" unirationality (unlike those properties
mentioned above).

\refstepcounter{equation}
\subsection{}
\label{subsection:int-5}

Let us start with some motivation for the forthcoming
constructions. Consider an elliptic fibration $f: X \map B$ over a
curve $B$ and suppose that $f$ admits a section. Form a group
$\frak{I}$ of all elliptic fibrations (over $B$) having $f: X\map
B$ their Jacobi fibration (see e.\,g.
\cite{shaf-principal-homog-spaces}). Set $E_b := f^{-1}(b)$ for
$b\in B$. Then there is a homomorphism
$$
h: \frak{I} \to \oplus_{b\in B}\text{Tors}(E_b)
$$
obtained from the fact that every multiple fiber of any elliptic
fibration from $\frak{I}$ is isogeneous to some $E_b$. Put
$\frak{I}_0 := \text{Ker}(h)$. This $\frak{I}_0$ is isomorphic to
the \emph{Brauer group $\bra(X) :=
H_{\text{\'et}}^2(X,\mathcal{O}_X^*)$ of the surface $X$}. This is
very much similar to $\bra(\fie(X))$ (or, as probably more to the
point, to $\bra_{\emph{v}}$) in {\ref{subsection:int-3}} and
suggests that the group $\bra(\fie(X))$ for any (or at least
sufficiently general in $\text{Hilb}_X$) variety $X$ should admit
an ``intrinsic" description in terms of the (birational) geometry
of $X$. Without going further into discussing the philosophy, let
us just say that we are up to the existence of rational fibrations
$f: X \dashrightarrow B$, where $B$ is a curve and generic fiber
of $f$ is birational to an Abelian variety. Call such $X$
\emph{b-\,Hamiltonian} (or \emph{b.-\,H.} for short). Then we
claim that being b.-\,H. is a stable birational invariant of $X$,
i.\,e. $X$ is b.-\,H. if (iff?) $X\times\p^k$ is so for some
(all?) $k\ge 1$. This setting, however, is too general to be true,
and we now pass to imposing some constraints on $X$ and giving the
precise statements.

\begin{remark}
\label{remark:artin-mum-clemens-griff} One may draw a parallel
between the above description of the ``Brauer-like" objects,
namely $\bra(\fie(X))$ vs. $\bra(X)$, and the ``algebraic"
approach of \cite{art-mum} vs. ``transcendental" methods of
\cite{cle-gri}, with considerations of both \cite{art-mum} and
\cite{cle-gri} boiling down to that in $H^3(X)$ (see also
\cite{dolg-gross}, \cite{murre}, \cite{colliot-oj}). Indeed, given
a conic bundle (or the \emph{Severi\,--\,Brauer}) structure $X
\map S$ (with certain restrictions on the discriminant), one
naturally assigns with $X$ a quaternion algebra $A$ (the Azumaya
algebra of rank $4$) over $\fie(S)$. This $A$ is a $2$-torsion
element in $\bra(\fie(S))$, which yields $H^4(X,\mu_2)$ and hence
$H^3_{tors}(X,\cel)\ne 0$, with the latter being an obstruction to
(stable) rationality of $X$. From the transcendental side in turn,
one considers $J(X) := H^3(X,\com)\slash (H^3(X,\cel)+H^{1,2})$
($\dim X = 3$), the \emph{intermediate Jacobian} of $X$, which is
an Abelian variety with a principal polarization $\Theta$. The
fundamental property of $J(X)$ is that when $X$ is rational, then
$(J(X),\Theta)\simeq\displaystyle\prod_i (J(C_i),\Theta_i)$ as
principally polarized Abelian varieties, where $C_i$ are some
smooth curves with Jacobians $J(C_i)$ and theta-polarizations
$\Theta_i$. Note however that the ``stable" version of $J(X)$ does
not make sense in the current setting (but compare with
Theorem~\ref{theorem:main} below).
\end{remark}

\refstepcounter{equation}
\subsection{}
\label{subsection:int-66}

Throughout the paper $A^n$ denotes an Abelian variety of dimension
$n$.

\begin{definition}
\label{definition:ab-var-rully} We call $A^n$ \emph{rully} if
there exists a sequence of Abelian subvarieties $0\varsubsetneq
A^1\varsubsetneq\ldots\varsubsetneq A^{n-1}\varsubsetneq A^n$.
\end{definition}

The subject of primary interest for us in the present narrative
will be the following:

\begin{conj}
\label{theorem:p-n-ab-pen} Projective space $\p^n$ is b.-\,H. for
any $n\ge 2$. More precisely, there exists a rational fibration
$\p^n\dashrightarrow\p^1$ with generic fiber $\approx A^{n-1}$ for
some rully Abelian variety $A^{n-1}$.
\end{conj}

Let $\mu_n$ be the cyclic group of order $n$. An elementary
manifestation of Conjecture~\ref{theorem:p-n-ab-pen} is the
rational surface $(E\times\p^1)\slash\mu_2$ fibred by elliptic
curves $\simeq E$ (cf. Example~\ref{example:halphen-pen} below).
At the same time, one can not extend the same construction
verbatim to the case of the quotient
$(E^n\times\p^1)\slash\mu_n^{\oplus n},~n\ge 2$, with
$\mu_n^{\oplus n}$ acting diagonally (and componentwise on $E^n$)
via $\pm$. However, we will reincarnate similar, ``geometric",
approach to Conjecture~\ref{theorem:p-n-ab-pen} later in
Section~\ref{section:misce}.

Let us now collect some examples and heuristics in support of
Conjecture~\ref{theorem:p-n-ab-pen} (the reader will find more
discussion in Section~\ref{section:misc}).

\begin{example}
\label{example:halphen-pen} Given an elliptic curve $E$ and nine
points $P_1,\ldots,P_9\in E$, the condition $m\sum_{i=1}^9 P_i =
0$ in the group $E(\com)$, for some $m\ge 1$, is equivalent to the
existence of a curve $Z\subset\p^2$ having $\deg Z = 3m$ and
$\text{mult}_{P_i}(Z) = m$ for all $i$. If $f=0, g=0$ are
equations of $E, Z$ respectively, then generic curve on $\p^2$
given by $\lambda f^m + \mu g = 0$, $\lambda,\mu\in\com$, is
birational to an elliptic curve. This is the example of a
\emph{Halphen pencil} on $\p^2$ (thus
Conjecture~\ref{theorem:p-n-ab-pen} is trivial when $n=2$) and in
fact any Halphen pencil ($=1$-dimensional linear system with
generic element birational to an elliptic curve) on $\p^2$ is
reduced to this one (for an appropriate $m$) via Cremona
transformations. This is a classical result
(\emph{Dolgachev\,--\,Bertini theorem}), proved in
\cite{dolg-bert}, which provides one with a huge source of
problems of similar type for other rationally connected varieties,
such as existence and (explicit) description of Halphen pencils
(i.\,e. generic hypersurface in the pencil is required to have
Kodaira dimension $0$) on these varieties (compare with
Conjecture~\ref{theorem:p-n-ab-pen} or
Theorem~\ref{theorem:vanya-ilya} below for instance).
Unfortunately, the case of $\p^2$ seems to be rather an exception
(a gem if one likes), since in general one should not expect
similar neat (or ``contact") description of special fibrations on
$\p^n$ say, $n\ge 3$. In the latter case for example, if $A^{n-1}$
is as in Conjecture~\ref{theorem:p-n-ab-pen}, then its image in
$\p^n$ is necessarily a non-normal (!) hypersurface, as one easily
elaborates via the Lefschetz-type theorem(s) (see e.\,g.
\cite[3.1]{laz}), together with \cite[Theorem 11]{kol-lars} and
\cite[Corollaries 1.3, 1.5]{hac-mc}.\footnote{We should mention
here a construction, which is due to J.\,Koll\'ar, justifying
Conjecture~\ref{theorem:p-n-ab-pen} in the case $n = 3$. Namely,
let $E$ be an elliptic curve with the $\mu_3$-complex
multiplication. Then the quotient $X := (E\times
E\times\p^1)\slash\mu_3$ by (non-trivial) diagonal $\mu_3$-action
is a rational $3$-fold with a map $X \dashrightarrow \p^1$ whose
general fiber $= E\times E$.}
\end{example}

To overcome the difficulty pointed out in
Example~\ref{example:halphen-pen}, we develop the above
observation with $(E\times\p^1)\slash\mu_2$ further. The idea is
to use (annoying) similarity between two ``torus-like" objects:
$A^n(\com) = \com^n\slash\cel^n$ and $(\com^*)^n$. We would like
to stress that this similarity between two different types of tori
becomes even stronger when the ground field $\fie = \com$ is
replaced by a field having $\text{char} = p > 0$. Namely, if
$\fie$ is the global field $\f_p(t)$ say, and $A^n$ is
\emph{ordinary} (see {\ref{subsection:misc-3}} below), we indicate
in Section~\ref{section:misc} that (after completing and closing
$\fie$ further) the underlying topological spaces $A^n(\fie)$ and
$\p^n(\fie)$ are \emph{homeomorphic}. This is done in a framework
very much typical to the one developed along the uniformization
theories for modular varieties (see Intermedia in
{\ref{subsection:misc-5}} for discussion and references). The only
difference is that instead of considering moduli (cf.
{\ref{subsection:int-5}}) we stick to the patching data inscribed
into the one solid $A^n$. This is a familiar type of duality one
meets in K\"ahler geometry for instance, where variation of
K\"ahler structure $(V,\omega),\omega\in H^2(V,\cel)$, on a
compact complex manifold $V$ may equally be seen in terms of the
action of the symplectomorphism group $\text{SDiff}(V,\omega)$ on
$V$ (see e.\,g. \cite{donaldson-kahler-hamiltonian}). Again we
postpone the discussion of further analogies until
Section~\ref{section:misc} and go on with applications of
Conjecture~\ref{theorem:p-n-ab-pen}.

\begin{remark}
\label{remark:manin-arakelov} Take the $3$-dimensional hyperbolic
space $\mathbb{H} := \com \times\mathbb{R}_{\ge0}$ acted
(isometrically) by a free subgroup $\Gamma\subset
\text{PSL}(2,\com)$ in $g\ge 2$ generators. The action of $\Gamma$
extends to the one on the ideal boundary of $\mathbb{H}$
identified with $\p^1(\com)$. Let $\Lambda_{\Gamma}\subset\p^1$ be
the closure of the attractive and repulsive fixed points for all
the elements $\gamma\in \Gamma$. Then for the complement
$\Omega_{\Gamma} := \p^1\setminus{\Lambda_{\Gamma}}$, the quotient
$X := \Omega_{\Gamma}\slash\Gamma$ is a Riemann surface of genus
$g$, and the covering $\Omega_{\Gamma} \map X$ is called
\emph{Schottky uniformization} of $X$ (in fact every (oriented)
Riemann surface admits a Schottky uniformization). Again, in line
with the previous discussion let us note that $X \subset
\mathbb{H}\slash/\Gamma$, a $3$-dim handlebody of genus $g$. We
refer to \cite{manin-arak} (especially, to a beautiful parallel
with the $p$-adic case, treated in \cite{mumford-curves-uni},
\cite{mumford-ab-vars-uni}) for further illustrations and
applications of this ``introducing-extra-dimension" principle.
This is another supplement for treating geometric objects over a
global field of $\text{char}>0$ (cf. {\ref{subsection:misc-5}}
below).
\end{remark}

\refstepcounter{equation}
\subsection{}
\label{subsection:int-6}

The main result actually proved in this paper is the following:

\begin{theorem}
\label{theorem:main} Assuming Conjecture~\ref{theorem:p-n-ab-pen},
suppose $X$ is \emph{stably rational}, i.\,e. $X \times \p^k
\approx \p^{n+1}$ for some $k$ and $n = \dim X+k-1$. Then one (or
both) of the following holds:

\begin{itemize}

\item $X$ is rationally fibred by hypersurfaces of negative
Kodaira dimension;

\smallskip

\item $X$ is b.-\,H.

\end{itemize}

\end{theorem}

Given Theorem~\ref{theorem:main}, it is tempting to produce
examples of non-stably rational rationally connected varieties, in
addition to linear quotients and conic bundles discussed
above.\footnote{Note that any cubic $3$-fold $X_3\subset\p^4$ is
b.-\,H., since $X \approx \p^3\slash\tau$ for a birational
involution $\tau\in\text{Bir}(\p^3)$, and so the arguments in the
proof of Corollary~\ref{theorem:main-cor-a} apply.} Let us first
mention the following result (see also
\cite{vanya-kod-0-in-vasyas-paper}, \cite{jihun-vanya}):

\begin{theorem}[{see \cite{isk-man}, \cite{corti}, \cite{vanya-ilya}}]
\label{theorem:vanya-ilya} Any smooth quartic $3$-fold
$X_4\subset\p^4$ is not b.-\,H. and can not be rationally fibred
by hypersurfaces of negative Kodaira dimension.
\end{theorem}

Theorems~\ref{theorem:main}, \ref{theorem:vanya-ilya} and
Conjecture~\ref{theorem:p-n-ab-pen} yield

\begin{cor}
\label{theorem:main-cor} Any smooth quartic $3$-fold
$X_4\subset\p^4$ is not stably rational.
\end{cor}

Theorem~\ref{theorem:main} is proved in Section~\ref{section:pro}
(see Remark~\ref{remark:illustr} for an outline). Let us indicate
that Theorem~\ref{theorem:main} relies mostly on Theorem 11 in
\cite{kol-lars} (and is similar, in a way, to Theorem 14 in
\emph{loc.\,cit.}) and on \cite[Theorem 11.1]{kollar-shaf}, and
these are the only (essentially) non-trivial results used in the
proof. We will return to the exposition of \cite{kol-lars} later
in Section~\ref{section:misce}.

\begin{remark}
\label{remark:rully-uni-rat} In the setting of
Conjecture~\ref{theorem:p-n-ab-pen}, it is easy to modify the
arguments of Section~\ref{section:pro} in such a way that the
assertion of Theorem~\ref{theorem:main} still holds, with $\p^n$
replaced by some \emph{unirational} b.-\,H. variety $\p$ with a
transitive regular action of the group $\aut(U)$ on a Zariski open
subset $U\subseteq\p$ (as one just needs irreducibility of
intersections of $x\times\p^k$, for various (generic) $x\in X$,
with (birationally) rully Abelian hypersurfaces in $X\times\p^k$).
\end{remark}

\begin{cor}
\label{theorem:main-cor-a} Let $X$ be as in
Question~\ref{question:colliot}. Then the same options as in
Theorem~\ref{theorem:main} hold in this case.
\end{cor}

\begin{proof}
Indeed, since $X$ is considered up to stable birational
equivalence, we may assume that $X \times \p^k = \frak{G}\slash H$
for some $k$, where $\frak{G}$ is a connected algebraic group and
$H\subset \frak{G}$ is a finite subgroup (see \cite{bkk},
\cite{bog-brauer-conn-lin-groups}). Then notice that
$\frak{G}\approx\p^n$ (see \cite{chev}), with $n\gg 1$, so that
one finds a $1$-parameter family $\{A_t\}_{t\in\com}$ of
(birationally) Abelian hypersurfaces in $\frak{G}$ (with $A_t\sim
A_t'$ (linearly equivalent) for all $t,t'$, as
Conjecture~\ref{theorem:p-n-ab-pen} predicts). Pick a generic
point $P \in \frak{G}$ and some $A_t\ni P$. It suffices to
establish that $(H\cdot P) \cap A_t = \{P\}$
(scheme-theoretically) for the $H$-orbit of $P$. Indeed, if this
is the case, then for generic $t'$ the hypersurface $A_{t'}$ maps
isomorphically into $\p: = \frak{G}\slash H$ under the quotient
morphism $q: \frak{G} \map \p$, and the cycles $q_*(A_t)$,
$q_*(A_{t'})$ are rationally equivalent. But $\p$ is rationally
connected, which implies that in fact $q_*(A_{t})\sim
q_*(A_{t'})$, and we are done by
Remark~\ref{remark:rully-uni-rat}.

Now, in order to achieve the identity $(H\cdot P) \cap A_t =
\{P\}$, we take an appropriate $\sigma\in\text{Bir}(\frak{G})$
such that $\sigma(P) = P$ and $\sigma(H\cdot
P\setminus{\{P\}})\cap A_t=\emptyset$ (cf.
{\ref{subsection:int-4}}). It remains to replace $A_t$ by
$\sigma^{-1}(A_t)$.
\end{proof}

The next example provides negative answer to
Question~\ref{question:colliot}:

\begin{example}
\label{example:neg-answ-colliot} The quartic $X_4\subset\p^4$ with
equation $x_0^4 + x_1^4 + x_2^4 + x_3^4 + x_0x_4^3 + x_1^3x_4 -
6x_2^2x_3^2 = 0$ is smooth and unirational (see e.\,g.
\cite{isk-man}). One concludes via
Corollary~\ref{theorem:main-cor-a} and
Theorem~\ref{theorem:vanya-ilya}.
\end{example}

\bigskip

\section{Proof of Theorem~\ref{theorem:main}}
\label{section:pro}

\refstepcounter{equation}
\subsection{}
\label{subsection:pre-1}

We use notations from {\ref{subsection:int-66}}. Throughout this
section $A^{n}$ is supposed to be rully. Recall that there is a
generically $1$\,-to\,-$1$ (onto its image) map from $A^{n}$ to
$\p^{n+1}$ (since Conjecture~\ref{theorem:p-n-ab-pen} holds by our
assumption). Then, composing the natural projection
$X\times\p^k\map X$ with a birational isomorphism
$X\times\p^k\approx\p^{n+1}$, we obtain a rational dominant map
$\phi: A^{n}\dashrightarrow X$.

\begin{lemma}
\label{theorem:fiber-irred} The general fiber of $\phi$ is
irreducible ($k\ge 2$).
\end{lemma}

\begin{proof}
Identify $A^n$ with a hypersurface in $X\times\p^k$. Then $\phi$
coincides with the restriction to $A^n$ of the projection
$X\times\p^k\map X$. The fibers of the induced morphism $A^n\map
X$ are all of the form $(\text{pt}\times\p^k)\cap A^n$. Thus it
suffices to show that $(\text{pt}\times\p^k)\cap A^n$ is
irreducible for generic $\text{pt}\in X$.

Recall that $A^n$ varies in a pencil $\mathcal{P}$ on
$X\times\p^k$. There is also a Zariski open subset $U\subseteq
X\times\p^k$ with a transitive regular $\aut(U)$-action. Then we
have $A^n\sim g_*A^n$ for all $g\in G$ and an open subgroup
$G\subseteq\aut(U)$. In particular, since $\dim G \gg 1$, there is
a surface $\Pi\subseteq |A^d|$ such that the line $\mathcal{P}$
varies in a $2$-dimensional family on $\Pi$. Hence $\Pi = \p^2$
and we may replace $\mathcal{P}\ni A^n$ by a net. Then (generic)
$A^n$ is smooth in codimension $1$, hence normal, and so $\phi$
has irreducible general fiber.
\end{proof}

Let us resolve the indeterminacies of $\phi$:
$$
\xymatrix{
&&Y\ar@{->}[ld]_{\sigma}\ar@{->}[rd]^{\sigma'}&&\\%
&A^n\ar@{-->}[rr]_{\phi}&&X.&}
$$
Here we take $\sigma$ to be composed of the blow-ups at smooth
centers. In particular, we have $K_Y = \sum a_iE_i$, where $E_i$
are $\sigma$-exceptional divisors and $a_i>0$ for all
$i$.\footnote{To be more precise, one has to include the case $a_j
= 0$ for some $j$ as well. But then $Y = A^n,~\sigma =
\text{id},~\phi$ is regular and it is easy to see that $\kappa(X)
= 0$, a contradiction.} Furthermore, taking generic
$A^{d+1}\subset A^n,~d:=\dim X$, we may stick to the case $n=d+1$
and so the general fiber $F$ of $\phi$ is an irreducible curve
(see Lemma~\ref{theorem:fiber-irred}).

\begin{lemma}
\label{theorem:almost-flat} There is a proper subscheme
$\widetilde{Y}\subset\p(\sigma'_*\omega_{Y/X})$ such that
$\widetilde{Y}\approx Y$ and $\widetilde{Y}\simeq Y$ near $F$.
\end{lemma}

\begin{proof}
Indeed, the curve $\sigma_*^{-1}F$ is smooth, l.\,c.\,i. and
canonically polarized by $K_Y = \sum a_iE_i$, which gives a
subscheme in $\p(\sigma'_*\omega_{Y/X})$ (restriction of the
Hilbert family) birational to $Y$.
\end{proof}

Let $\widetilde{Y}$ be as in Lemma~\ref{theorem:almost-flat} and
$\tilde{\sigma}: \widetilde{Y}\map \widetilde{X}$ be the induced
(flat) morphism ($\widetilde{X}\approx X$). We may replace
$\widetilde{X}$ by its resolution so that $\tilde{\sigma}$ remains
flat. We will also suppose $\widetilde{Y}$ to be normal and CM
(the general case can be reduced to this setting).

\begin{lemma}
\label{theorem:contract} The induced birational map
$\widetilde{Y}\dashrightarrow Y\stackrel{\sigma}{\map} A^n$ is
contracting.\footnote{We refer to \cite{kollar-mori} for standard
notions and facts from the minimal model theory.}
\end{lemma}

\begin{proof}
We may assume w.\,l.\,o.\,g. there is a birational morphism $f:
Y\map\widetilde{Y}$. Note that
$$
K_{\widetilde{Y}} \equiv \sum a_iE'_i + \sum b_jF_j
$$
(numerically) on $\widetilde{Y}$ for some $b_j\in\ra$ with $F_j
\cdot (f\circ\sigma)_*^{-1}F = 0$ for all $j$. Here $E'_i$ are the
proper transforms of (some of the) $E_i$ from the above formula
for $K_Y$ ($a_i$ are the same). Further, since $\sigma_*K_Y = 0$
and $K_Y - f^*(K_{\widetilde{Y}})\subseteq\text{Exc}(f)$, we
obtain that $f_*^{-1}F_j = E'_i$ for some $i = i(j)$ and every $j$
whenever $b_j\ne 0$. Hence we may assume that $b_j = 0$ for all
$j$. This implies that $f_*K_Y = K_{\widetilde{Y}}$ and so
$\widetilde{Y}$ has rational singularities. The result now follows
from the weak factorization theorem for birational morphisms (as
$A^n$ does not contain rational curves).
\end{proof}

\refstepcounter{equation}
\subsection{}
\label{subsection:pre-132434}

Consider the morphism $\pi: Y\stackrel{\sigma}{\map} A^n\map
A^n/A^{d-1}$. Put $S := \sigma'^{-1}\sigma'(\Omega)\subset Y$ for
$\Omega := \sigma_*^{-1}(A^{d-1})$. Note that $\Omega$ is a fiber
of $\pi$. It follows from Lemma~\ref{theorem:contract} that the
induced map $\pi\circ f^{-1}: \widetilde{Y}\dashrightarrow
A^n/A^{d-1}$ is defined near $f_*\Omega$ for generic (varying)
$A^{d-1}$. This allows us restrict to the case $\widetilde{Y} = Y$
(the arguments below work literarilly in the general setting).

\begin{lemma}
\label{theorem:base-point-free} The linear system $|S|$ is
basepoint-free on $Y$.
\end{lemma}

\begin{proof}
By construction, the cycles $\Omega$ (with varying $A^{d-1}$) are
all algebraically equivalent on $Y$, since $\Omega$ is a fiber of
$\pi$. Then the cycles $\sigma'_*(\Omega)$ on $X$ are also
algebraically equivalent. Moreover, since $X$ is smooth and
rationally connected, thus having Alb$(X) = 0$, all
$\sigma'_*(\Omega)$ are in fact linearly equivalent divisors. In
particular, since $\sigma'$ is flat, the divisors
$\sigma'^{-1}\sigma'(\Omega)\subset Y$ are linearly equivalent as
well. Finally, since all $\sigma'^{-1}\sigma'(\Omega)$ belong to
$|S|$ and $A^{d-1}$ was chosen arbitrarily, the assertion follows.
\end{proof}

Choose $\Sigma\in |S|$ generic. Note that $\Sigma$ is smooth.
Also, blowing up $X$ and making a base change for $\sigma'$, we
may assume that $\dim |S| = 1$

\begin{prop}
\label{theorem:kapa-0} The hypersurface $\Sigma$ contains a fiber
of the morphism $\pi\big\vert_{\Sigma}: \Sigma\map A^n/A^{d-1} =:
A^2$ birational to $\Omega$.
\end{prop}

\begin{proof}
Suppose not. Then all fibers of $\pi\big\vert_{\sigma(\Sigma)}$
have dimension $\le d-2$. Cutting with $A^3\subset A^n$ we reduce
to the case $\Sigma$ is a smooth surface such that
$\pi\big\vert_{\sigma(\Sigma)}: \sigma(\Sigma)\map A^2$ is finite.

\begin{lemma}
\label{theorem:kapa-0-a} Let $K_{\Sigma}$ be not $\sigma$-nef.
Then $X$ is (birationally) fibred by hypersurfaces of negative
Kodaira dimension.
\end{lemma}

\begin{proof}
Notice that $\sigma\big\vert_{\Sigma}$ is a resolution of
indeterminacies of the map $\phi\big\vert_{\sigma(\Sigma)}$. Also,
since $\Sigma\in |S|$ is generic, the intersections $\Sigma\cap
E_i$ are irreducible for all $i$. Let $C\subset\Sigma$ be a
$(-1)$-curve contracted by $\sigma$. Then we have $C\subseteq E_j$
for some $j$ and $-1=(C^2) = mE_j\cdot C$ for the (infinitely
close) multiplicity $m=1$ of $\sigma(\Sigma)$ (w.\,r.\,t. $E_j$).
In particular, for $\sigma(C)$ is a base point of
$\phi\big\vert_{\sigma(\Sigma)}$, the curve $C$ is a multisection
of the fibration $\sigma'\big\vert_{\Sigma}$ and so $\phi(C)$ is a
rational curve on $X$.

Thus, if $d = 2$, then there is a rational fibration on $X$ by
rational curves $\phi(\Sigma)$. (Note that at this point we do not
need any assumption on the map $\pi\big\vert_{\sigma(\Sigma)}$.)
Further, varying $A^3\subset A^n$ in our initial setting (with
$\dim\Sigma = d$) and assuming $K_{\Sigma\cap A^3}$ to be not
$\sigma$-nef for (most of) these $A^3$, we find a bunch of
rational curves $\phi(\Sigma\cap A^3)$ which cover the
hypersurface $\phi(\Sigma)$. This yields a fibration on $X$ as
claimed.
\end{proof}

Let us now consider the case of $K_{\Sigma}$ being $\sigma$-nef.
Here we use the assumption on $\pi\big\vert_{\sigma(\Sigma)}$ from
the beginning. We get
\begin{equation}
\nonumber K_{\Sigma} = \sigma^*K_{\sigma(\Sigma)} - \sum
b_jE_j\big\vert_{\Sigma}.\footnote{We assume for simplicity that
$\sigma(\Sigma)$ is normal and $\ra$-Gorenstein.}
\end{equation}
Here $E_j\big\vert_{\Sigma}$ are
$\sigma\big\vert_{\Sigma}$-exceptional so that $b_j\ge 0$ for all
$j$. On the other hand, we have $K_{\Sigma} = \sum
a_iE_i\big\vert_{\Sigma} = K_Y\big\vert_{\Sigma}$, which gives
\begin{equation}
\label{k-sigma-eq} \sum (a_i+b_i)E_i\big\vert_{\Sigma} =
\sigma^*K_{\sigma(\Sigma)}
\end{equation}
(with $b_i=0$ for non-$\sigma\big\vert_{\Sigma}$-exceptional
$E_i\big\vert_{\Sigma}$). Furthermore, the Hurwitz formula gives
$$K_{\sigma(\Sigma)} = \pi^*(D)$$ for some effective cycle
$D\subset A^2$, the ramification locus of $\pi\big\vert_{\Sigma}$.

\begin{lemma}
\label{theorem:kapa-0-b-14453} If the cycle
$E_j\big\vert_{\Sigma}$ is not
$\sigma\big\vert_{\Sigma}$-exceptional, then
$E_j\big\vert_{\Sigma}=\pi^*[\text{0-cycle on}\ A^2]$, all $j$.
\end{lemma}

\begin{proof}
Notice that $\sigma(S)$ is smooth near $\sigma(\Omega)$. This
gives $(\sigma(\Sigma)\big\vert_{\sigma(S)})\cdot\sigma(\Omega)=0$
and $\sigma(\Omega)^2 = 0$. Indeed, we have $K_{\sigma(\Sigma)} =
\pi^*(D)$, $\sigma(\Sigma)\big\vert_{\sigma(S)}=K_{\sigma(S)}$
near $\sigma(\Omega)$, and so $K_{\sigma(S)}\equiv S \cdot
\pi^*[\text{cycle on}\ A^2]$ (as $S\sim\Sigma$). In particular,
$K_{\sigma(S)}\cdot\sigma(\Omega)=0$ (for $\Omega$ is a fiber of
$\pi$), thus
$(\sigma(\Sigma)\big\vert_{\sigma(S)})\cdot\sigma(\Omega)=0$ and
$\sigma(\Omega)^2 = 0$.

Now $\sigma(E_j)$ equals (set-theoretically)
$\pi^{-1}[\text{0-cycle on}\ A^2]$ and the result follows.
\end{proof}

It follows from Lemma~\ref{theorem:kapa-0-b-14453} that the left
hand side of \eqref{k-sigma-eq} is negative on all the
$E_j\big\vert_{\Sigma}$ that are
$\sigma\big\vert_{\Sigma}$-exceptional. This shows that $b_j = 0$
for all $j$.

\begin{lemma}
\label{theorem:kapa-0-b-1432} The rational map
$\phi\big\vert_{\sigma(\Sigma)}$ is everywhere defined.
\end{lemma}

\begin{proof}
Indeed, given $b_j = 0$ for all $j$ there are either no
$\sigma\big\vert_{\Sigma}$-exceptional curves
$E_j\big\vert_{\Sigma}$, or all of them are $(-2)$-curves on
$\Sigma$. We are done in the former case. In the latter case,
arguing as in the proof of Lemma~\ref{theorem:kapa-0-a} we find
that $E_j\big\vert_{\Sigma}$ are multisections of
$\sigma'\big\vert_{\Sigma}$, thus $X$ is (birationally) fibred by
hypersurfaces of negative Kodaira dimension.
\end{proof}

According to Lemma~\ref{theorem:kapa-0-b-1432} we may identify
$\sigma(\Sigma)$ with its normalization $\Sigma$ (and
$\phi\big\vert_{\sigma(\Sigma)}$ with
$\sigma'\big\vert_{\Sigma}$). Suppose that the ramification cycle
$D$ is big (i.\,e. $D^2>0$).

\begin{lemma}
\label{theorem:kapa-0-b} $\phi(\Sigma) \simeq A^1$.
\end{lemma}

\begin{proof}
Recall that $\phi\big\vert_{\Sigma}: \Sigma\map X$ is a flat
morphism with connected $1$-dimensional fibers. Let $|F|$ be the
corresponding linear system on $\Sigma$. Note that $F = \pi^*(E)$
for an elliptic curve $E\subset A^2$ (as $(F^2)=0$ and $D\cap
E\ne\emptyset$). This gives a morphism $\Sigma\map A^2\map
A^2\slash E =: A^1$ with connected fibers $\sim F$. But this means
that $\phi(\Sigma) \simeq A^1$.
\end{proof}

Further, cutting with (varying) $n-3$ hypersurfaces from
$\phi^*\mathcal{O}_{X}(1)$ in our initial setting (with
$\dim\Sigma = d$) and applying the arguments in the proof of
Lemma~\ref{theorem:kapa-0-b} to the morphism $\Sigma\map A^d\map
A^d\slash E =: A^{d-1}$ (for again $K_{\Sigma} = \pi^*(D)$, $D$ is
big, etc.), we obtain that $\phi(\Sigma)\approx A^{d-1}$. Thus $X$
is b.-\,H. in this case.

Now let $D^2=0$. Then $(K_{\Sigma}^2)=0$ and the linear system
$|K_{\Sigma}| = |\pi^*(D)|$ is basepoint-free.

\begin{lemma}
\label{theorem:kapa-0-c} The cycle $S\cdot\Sigma\subset A^n$ is
supported on the fibers of $\phi\big\vert_{\Sigma}$.
\end{lemma}

\begin{proof}
We have $S\cdot\Sigma\in |K_{\Sigma}|$ and also
$S\cdot\Sigma\supseteq F$ (see the proof of
Lemma~\ref{theorem:kapa-0-b}). But $S\cdot\Sigma$ is a disjoint
union of elliptic curves $C_i$ by the assumption on $K_{\Sigma}$.
The assertion now follows because $(F^2) = C_i\cdot F = 0$ for all
$i$.
\end{proof}

Again, cutting with (varying) $n-3$ hypersurfaces from
$\phi^*\mathcal{O}_{X}(1)$ in our initial setting (with
$\dim\Sigma = d$) and applying Lemma~\ref{theorem:kapa-0-c}, we
get $\sigma(\Sigma)^n = 0$. The latter means that
$\sigma'\big\vert_{\Sigma}$ factors through the morphism $A^n\map
A^n\slash A^1$. This also holds for $\sigma'\big\vert_S$ (as
$S\sim\Sigma$). Furthermore, since $\Omega\subset S$, the morphism
$\sigma'\big\vert_S$ admits a (rational) section. This gives
$S\approx\sigma(S)\simeq A^1\times A^{d-1}$. Then $S$ maps onto a
curve under the morphism $\pi$. Hence the same holds for
$\Sigma\in |S|$.

Proposition~\ref{theorem:kapa-0} is proved.
\end{proof}

We may assume that $\Omega\subset \Sigma$ due to
Lemma~\ref{theorem:kapa-0}. Now, if $\kappa(\sigma'(\Omega)) =
-\infty$, then $X$ is (birationally) fibred by hypersurfaces of
negative Kodaira dimension.

\begin{prop}
\label{theorem:final-lem} Let $\kappa(\sigma'(\Omega)) \ge 0$.
Then $\sigma'(\Omega)\approx$ an Abelian variety.
\end{prop}

\begin{proof}
Notice first that $\kappa(\sigma'(\Omega)) = 0$ because
$\kappa(\Omega)\ge\kappa(\sigma'(\Omega))$ for generically finite
dominant $\sigma'\big\vert_{\Omega}: \Omega \map \sigma'(\Omega)$.
Also, since $\sigma'$ is flat with $1$-dimensional fibers, the
morphism $\sigma'\big\vert_{\Omega}$ is finite. Then, for
$\sigma'(\Omega)$ is smooth (cf.
Lemma~\ref{theorem:base-point-free}) and not uniruled, the result
of \cite[Theorem 11]{kol-lars} implies that
$\sigma'\big\vert_{\Omega}$ is \'etale.\footnote{Note that
$\sigma'\big\vert_{\Omega}$ is not ramified along the
$\sigma\big\vert_{\Omega}$-exceptional locus by generality of
$\Sigma\supset\Omega$.} This yields the estimate
$$d-1 = q(\Omega)\le q(\sigma'(\Omega))$$ for irregularities.

On the other hand, the Leray spectral sequence (applied to
$\sigma'\big\vert_{\Sigma}$) shows that
\begin{equation}
\label{eq-1} q(\Sigma) = q(\sigma'(\Omega)) + g,
\end{equation}
where $g$ is the genus of a fiber of $\sigma'\big\vert_{\Sigma}$.
Furthermore, restricting the morphism $\pi:
Y\stackrel{\sigma}{\map} A^n\map A^n/A^{2} =: A^{d-1}$ to
$\Sigma$, we obtain a morphism $f: \Sigma\map A^{d-1}$. Note that
the fibers of $f$ are connected because $\Omega\subset\Sigma$ is
(generically) its section. In particular, we get
\begin{equation}
\label{eq-2} q(\Sigma) = q(\Omega) + g',
\end{equation}
where $g'$ is the genus of a fiber of $f$, and so $g\le g'$.

\begin{lemma}
\label{theorem:final-lem-op-est} The estimate $g\ge g'$ holds.
\end{lemma}

\begin{proof}
We may replace $\sigma'(\Omega)$ by its \'etale cover $\Omega$,
make a base change for $\sigma'\big\vert_{\Sigma}$, reduce to the
case of $\sigma'\big\vert_{\Sigma}$ being a fibration over
$\Omega$ (with connected fibers and a section), and thus arrive at
$q(\Sigma) = q(\Omega) + g$ (in place of \eqref{eq-1}). Then
\eqref{eq-2} turns into $q(\Sigma) = q(\Omega) + (g'-1)m + 1$ for
some $m\ge 1$. This gives $g\ge g'$.
\end{proof}

Lemma~\ref{theorem:final-lem-op-est} implies that
$q(\sigma'(\Omega)) = d-1$ as well. The assertion now follows from
\cite[Theorem 11.1]{kollar-shaf}.
\end{proof}

Proposition~\ref{theorem:final-lem} finishes the proof of
Theorem~\ref{theorem:main}.

\begin{remark}
\label{remark:illustr} It is instructive to illustrate the
preceding arguments on a model example, that is of $X := \p^2$,
$A^n := A^3$. One proceeds along the following steps:

\begin{itemize}

    \item start with a rational map $\phi: A^3 \dashrightarrow X$ with
    irreducible fibers, resolve $\phi$ via $\sigma: Y\map A^3$ as in
    {\ref{subsection:pre-1}}, arrive at a regular fibration
    $\sigma': Y \map X$ and reduce to the case when $\sigma'$ is
    flat (see the beginning of {\ref{subsection:pre-132434}});

    \smallskip

    \item each surface $S = \sigma'^{-1}\sigma'(\Omega)$ carries
    an elliptic curve $\Omega\subset Y$ which is a multisections of the
    fibration $\sigma'\big\vert_S$;

    \smallskip

    \item as $\sigma'$ is flat, the linear system $|S|$ is the
    pullback of a linear system on $X$, and we may then assume
    $|S|$ to be a basepoint-free pencil (as before Proposition~\ref{theorem:kapa-0});

    \smallskip

    \item for general $\Sigma\in |S|$, we observe that if the
    divisor
    $K_{\Sigma}$ is not $\sigma\big\vert_{\Sigma}$-nef, then
    $\phi(\Sigma)$ is a rational curve, which gives a pencil of
    rational curves on $X$ (see Lemma~\ref{theorem:kapa-0-a});

    \smallskip

    \item in turn, if $K_{\Sigma}$ \emph{is}
    $\sigma\big\vert_{\Sigma}$-nef and the factorization map $A^3\map A^3\slash
    A^1 = A^2$ (where $A^1 = \sigma(\Omega)$) induces a \emph{finite} morphism $\sigma(\Sigma)\map
    A^2$, then the map
    $\phi\big\vert_{\sigma(\Sigma)}$ is everywhere defined, and we
    may assume that $\sigma(\Sigma)=\Sigma$ (see
    Lemmas~\ref{theorem:kapa-0-b-14453}, \ref{theorem:kapa-0-b-1432});

    \smallskip

    \item furthermore, if the above map $\sigma(\Sigma)\map
    A^2$ has \emph{big} ramification locus, then we obtain (Lemma~\ref{theorem:kapa-0-b}) that
    $\phi(\Sigma)$ is an elliptic curve, and so $X$ is b.-\,H. in this case;

    \smallskip

    \item in the remaining cases, the surfaces $\Sigma\in |S|$ carry
    elliptic curves $\equiv \Omega$ on $Y$, and these give either
    a
    rational or an elliptic fibration on $X=\p^2$ (cf.
    Proposition~\ref{theorem:final-lem}).

\end{itemize}

\end{remark}

\bigskip

\section{Evidence for Conjecture~\ref{theorem:p-n-ab-pen}: an algebraic approach}
\label{section:misc}

\refstepcounter{equation}
\subsection{Elliptic curves}
\label{subsection:misc-1}

The most direct way of relating $A^d$ to $\p^d$ is via the
\emph{uniformization} of $A^d$. Namely, one introduces a number of
(transcendental) parameters, having their source in the affine
space $\fie^d$ and the target in $A^d$. For instance, classically
one has $A^d(\com) = \com^d/\Lambda$, where $\Lambda\subset\com^d$
is a full sublattice. We are up to a similar picture over (almost)
arbitrary ground field $\fie$. In what follows, we will be mostly
concerned with the case $d = 1$, that is of $A^1 =: E$ being an
elliptic curve (this should suffice to develop the necessary
intuition for any $A^d$).

Recall that $E$ is given by an equation (in the Weiestrass normal
form)
$$
y^2 = 4x^3 - g_2x - g_3
$$
on the $(x,y)$-plane, with some $g_2,g_3\in\fie$, where we suppose
$p:=\text{char}~\fie\ne 2,3$. Adding a new variable $z$, we
identify $E$ with a cubic curve in $\p^2$ (the point $[0:1:0]\in
E$ being $0$ w.\,r.\,t. the group law on $E(\bar{\fie})$). Let
$\widehat{E}$ be the completion of $E$ at $0$ (cf.
Example~\ref{example:ab-var-form} below). This is simply the
formal scheme $\text{Spec}~\fie[[t]]$ for a $\fie$-transcendental
variable $t$. Furthermore, the group law $E\times E\map E$ leads
to a Hopf algebra structure on $\fie[[t]]$, thus making
$\widehat{E}$ into a \emph{formal} (aka \emph{local, quasi, etc.)
group} (see e.\,g. \cite{lubin-tate}, \cite{manin-fg-big-paper},
\cite{serre-elliptic-curves},
\cite{oort-formal-groups-newton-polygons}, as well as
{\ref{subsection:misc-2}} below, for an account of the basic
theory).

More specifically, letting
$$
\tilde{x} := -\frac{2x}{y},\ \tilde{y} := -\frac{2}{y}
$$
in the above discussion we get $E$ defined by the equation
$$
\tilde{y} = \tilde{x}^3 + \frac{g_2}{4}~\tilde{x}\tilde{y}^2 +
\frac{g_3}{4}~\tilde{y}^3
$$
on the $(\tilde{x},\tilde{y})$-plane, and so
$$
\tilde{y} = \tilde{x}^3(1 + A_1\tilde{x} + A_2\tilde{x}^2 +
\ldots),
$$
where $A_i$ are homogeneous polynomials in $\cel[g_2/4,g_3/4]$. On
the other hand, for any two points $P_1,P_2\in E,~P_i :=
(x_i,y_i),~x_i,y_i\in\fie$, one gets the $x$-coordinate of the
point $P_3 := P_1+P_2$ in the form
$$
-x_1 - x_2 + \frac{1}{4}\big(\frac{y_1-y_2}{x_1-x_2}\big)^2.
$$
Now, substituting the preceding (formal) series expression for
$\tilde{y}_i$, we find that the formal group law on $\widehat{E} =
\fie[[t]]$ is given by an element $G(x_1,x_2)\in \fie[[x_1,x_2]]$,
which carries the following properties:

\renewcommand{\theenumi}{$G_{\arabic{enumi}}$}

\begin{enumerate}

\item\label{eq-g-1} $G(x_1, G(x_2,x_3)) = G(G(x_1,x_2),x_3)$
    (associativity),

    \smallskip

\item $G(x_1,x_2) = G(x_2,x_1)$ (commutativity),

\smallskip

\item $G(x,i_G(x)) = 0$ for a unique $i_G(x)\in\fie[[x]]$ (existence
of an inverse),

\smallskip

\item\label{eq-g-4} $G(x,0) = G(0,x) = x$ (i.\,e. $x=0$ is the identity w.\,r.\,t. $G$).

\end{enumerate}
The series $G(x_1,x_2)$ is called \emph{formal (abelian, over
$\fie$) group} (see {\ref{subsection:misc-2}} for further
recollections).

\begin{remark}
\label{remark:heur-strat} It is $\widehat{E}$ (and $G(*,*)$) that
we are going to use as a source of ``uniformizers" for $E$. But
before gluing $E$ out of $\widehat{E}~$'s (= constructing entire
(multi-valued) functions on $E(\bar{\fie})$), one has to equip $E$
with extra symmetries, in line with the complex-analytic case of
$E(\com)$ above. For instance, heuristically at least, one has to
have a counterpart of the complex conjugation on the set
$E(\bar{\fie})$ in order to be able to speak about ``entire" (or
``holomorphic") parametrization of $E$. The latter forces the
topology of $\fie$ to be fruitful as well. We will develop all
these matters starting from {\ref{subsection:misc-3}}.
\end{remark}

\refstepcounter{equation}
\subsection{Formal groups}
\label{subsection:misc-2}

Here we recall some standard definitions and facts about (abelian)
formal groups. Let us begin with two examples:

\begin{example}
\label{example:mult-form-gr} Take the polynomial $G(x,y) :=
x+y+xy\in\fie[x,y]$ and add up any two elements $a,b\in \fie$ via
$a+_{\scriptscriptstyle G}b := G(a,b)$. This defines the
\emph{multiplicative formal group} $G_m :=
(\fie,+_{\scriptscriptstyle G})$ (cf.
(\ref{eq-g-1})~--~(\ref{eq-g-4}) above). One can easily check the
series $\phi(x) := \log(1 + x) = x - \displaystyle\frac{x^2}{2} +
\ldots$ gives an isomorphism between $G_m$ and the \emph{additive
group} $G_a := (\fie,+)$ (note that the latter can also be
regarded as a formal group with the group law $G(x,y) := x + y$).
More generally, we set $G_m^{\oplus d} := (k,+_{\scriptscriptstyle
G_1})\times\underbrace{\ldots}_{\text{d
times}}\times~(k,+_{\scriptscriptstyle G_d})$, with $G_i(x,y) =
x+y+xy$ for all $i$, and call this formal group \emph{direct sum
of $d$ copies of $G_m$}.
\end{example}

\begin{example}
\label{example:ab-var-form} Construction of $\widehat{E}$ in
{\ref{subsection:misc-1}} can be extended to the case of an
arbitrary Abelian variety $A^n$. In fact, replacing $A^n$ by its
formal neighborhood at $0$ one gets a formal group
$\widehat{A^n}$. The latter is defined similarly to $G(x,y)$ in
(\ref{eq-g-1})~--~(\ref{eq-g-4}), with $x$ replaced by the
$n$-string $(x_1,\ldots,x_n)$ (the same with $y_i$ for $y$) and
$G$ replaced by the $n$-string
$(G_1,\ldots,G_n),G_i\in\fie[[x_1,\ldots,x_n,y_1,\ldots,y_n]]$.
Thus $\widehat{A^n}$ is a counterpart of the local group of a
($\com$- or $\mathbb{R}$-) Lie group.
\end{example}

Regard the affine space $\mathbb{A}^n$ as its formal completion at
the origin and fix two formal group structures
$G_1(x,y),~G_2(x,y)$ on $\mathbb{A}^n$ (as defined in
Example~\ref{example:ab-var-form}).

\begin{definition}
\label{definition:formal-group-hom} A formal endomorphism $f:
\mathbb{A}^n \longrightarrow \mathbb{A}^n$, $f :=
(f_1,\ldots,f_n)$, $f_i\in\fie[[x]]$, is called
\emph{homomorphism} of formal groups $G_i$ (specifically, from
$G_1$ to $G_2$) if $f(G_1(x, y)) = G_2(f(x),f(y))$. Further, $f$
is an \emph{isomorphism}, and we put $G_1 \stackrel{f}{\simeq}
G_2$ (or simply $G_1\simeq G_2$ (or $G_1\stackrel{\sim}{\to} G_2$)
if no confusion is likely), if there exists an inverse $f^{-1}$
($f\circ f^{-1} = f^{-1}\circ f = \text{id}_{\fie^n}$) which is
also a homomorphism. For $G_1 = G_2 =: G$, we denote by
$\text{End}(G)$ the set of all group endomorphisms $G\to G$, with
$\text{Aut}(G)\subset\text{End}(G)$ being the collection of all
automorphisms.
\end{definition}

If $p = \text{char}~\fie>0$, then according to \cite{dieud} after
a (formal) coordinate change on $G_1$ and $G_2$, respectively, one
can write $f_i = x_i^{p^k}$, $n - s_k < i \leq n - s_{k+1}$, for
some $0\leq s_m \leq\ldots\leq s_1\leq s_0 := n$, $m \in \cel_{\ge
0}$.

\begin{definition}
\label{definition:formal-group-isog} $f$ is called an
\emph{isogeny} if $s_m = 0$.
\end{definition}

\begin{example}
\label{example:frob} In the previous notations, let $G_2 :=
G_1^{~(p^k)}$, $k \geq 1$, be a formal group with the composition
law $G_2(x,y)$ obtained by first replacing $G_1(x,y)$ by
$(G_1(x,y))^{p^k}$ and then substituting $x_i,y_i$ in place of
$x_i^{p^k},y_i^{p^k}$, respectively, for all $1\le i\le n$. Then
$f := \text{Fr}_p^k$, with $f_i := x_i^{p^k}$ for all $i$, is an
isogeny (the $k^{\text{th}}$ iteration of the Frobenius
$\text{Fr}_p$). Note that $G_2$ (as a formal scheme) carries the
same underlying topological space as $G_1$, but the structure
sheaf $\mathcal{O}_{G_2}$ (as a sheaf of ${\bf k}^{p^k}$-algebras)
is identified with $\mathcal{O}_{G_1}^{p^k}$(as a sheaf of ${\bf
k}$-algebras).
\end{example}

In view of Definition~\ref{definition:formal-group-isog} and
Example~\ref{example:frob}, the groups $G_1,G_2$ are called
\emph{isogeneous} (and we put $G_1\sim G_2$) if there exists an
isogeny $f: G_1 \longrightarrow \text{Fr}_p^k(G_2) = G_2^{~(p^k)}$
for some $k$. Notice that $\sim$ is an equivalence relation.

\begin{remark}
\label{remark:formal-group-mult-by-p} Replacing $\fie^n$ by its
dual, one may associate the \emph{dual group} $G^*$ to the formal
group $G$. Consider $\text{Fr}_p$ acting on $G^*$, and its
conjugate $\text{Fr}^*_p$ acting on $G$. The group $G$ is thus
equipped (canonically) with the \emph{multiplication-by-$p$
endomorphism} $p\cdot\text{Id}_{G} :=
\text{Fr}_p\circ\text{Fr}^*_p = \text{Fr}^*_p\circ\text{Fr}_p$
(see \cite[Proposition 1.4]{manin-fg-big-paper}). For instance, if
$G = \widehat{A^n}$ (see Example~\ref{example:ab-var-form}), then
$p\cdot\text{Id}_G$ is the localization of the usual
multiplication-by-$p$ endomorphism of $A^n$. Note that in the
latter case $p\cdot\text{Id}_G$ is obviously an isogeny -- la
propri\'et\'e incontournable in the classification of formal
groups (cf. Theorem~\ref{theorem:manin-wide-a-structure} below).
\end{remark}

\refstepcounter{equation}
\subsection{Ordinary Abelian varieties}
\label{subsection:misc-3}

Let $\fie$ be as above ($p = \text{char}~\fie>0$). Recall that an
Abelian variety $A := A^d$ (over $\fie$) is called
\emph{$p$-ordinary} (or simply \emph{ordinary} if no confusion is
possible) if one of the following equivalent conditions holds:

\begin{itemize}

    \item $A$ contains $p^d$ points of order $p$;

    \item the Hasse\,--\,Witt matrix $\text{Fr}_p^*:
    H^1(A^{(p)},\mathcal{O}_{A^{(p)}})\to H^1(A,\mathcal{O}_A)$ is
    invertible.

\end{itemize}

It is impossible to give here an extensive account of the
beautiful theory of ordinary Abelian varieties (and the interested
reader is addressed to the papers \cite{deligne-ord-ab-vars},
\cite{katz}, \cite{ogus}, \cite{serre-tate-ord-ab} for which we
have tried to be in line with in our current exposition). That is
why we restrict ourselves to simply recalling some basic technical
facts which will be used further (see e.\,g.
{\ref{subsection:misc-5}} and the end of
{\ref{subsection:misc-7}}).

Let us start with the case $d = 1$:

\begin{example}
\label{example:hasse} Consider an elliptic curve $E$ given by
equation $y^2 = x(x-1)(x-\lambda)$ (on the affine piece $(z\ne 0)$
of $\p^2 = \text{Proj}~\f_p\left[x,y,z\right]$) for some
$\lambda\in\f_p\setminus{\{0,1\}}$. If $E$ is not ordinary, then
it is called \emph{supersingular} -- the property characterized by
the vanishing $h_E :=
\displaystyle\sum_{i=0}^{m}\binom{m}{i}^2\lambda^i = 0$, where $m
:= (p-1)/2$ (see \cite[Corollary 4.22]{har} for example). In
particular, once $E$ is the mod $p$ reduction of an elliptic curve
$\widetilde{E}$ defined over a number field $K$, there is a place
$\frak{p}\subset K$ such that $\widetilde{E}$ admits an ordinary
reduction modulo $\frak{p}$. Note also that if $p\ne q$ are primes
and $E,E'$ are elliptic curves over $\ra$ say, so that
$h_{E}\equiv 0$ mod $p$ but $h_{E}\not\equiv 0$ mod $q$, and
similarly for $E'$ with $p,~q$ interchanged (we suppose the mod
$p$ and $q$ reductions of both $E,E'$ are non-singular), then $E$
and $E'$ are \emph{not} isogeneous (cf.
Definition~\ref{definition:formal-group-isog}).
\end{example}

\begin{remark}
\label{remark:ord-red} The setting of Example~\ref{example:hasse}
can be generalized as follows. Let $A$ be as above (not
necessarily ordinary). Suppose $A$ is the mod $\frak{p}$ reduction
of an Abelian variety $\widetilde{A}$ defined over some number
field $K$ (with $\frak{p}\cap\cel = (p)$). Then, applying
\cite[Corollary 2.8]{ogus}, one may choose $\frak{p}$ (after
enlarging $K$ if necessary) to be such that $A$ is ordinary.
\end{remark}

Discussion started in Remark~\ref{remark:formal-group-mult-by-p}
finds its further development in the following fundamental

\begin{theorem}[{see \cite[Proposition 2, Theorem 2]{manin-2}}]
\label{theorem:manin-wide-a-structure} Let $A^d$ be a $p$-ordinary
Abelian variety over $\fie$. Then
$$
\widehat{A^d} \sim G_m^{\oplus d}
$$
(cf. Example~\ref{example:mult-form-gr}). More precisely, there is
an isogeny $\Theta: \widehat{A^d} \longrightarrow G_m^{\oplus d}$,
defined up to Frobenius twist, such that $\Theta\circ
p\cdot\text{Id}_{\widehat{A^d}} = \text{Fr}_{p}\circ \Theta$ (cf.
Example~\ref{example:frob}).
\end{theorem}

\begin{remark}
\label{remark:ordin-ab-vars-uniform}
Theorem~\ref{theorem:manin-wide-a-structure} may be considered as
the first crucial manifestation of a ``torus-like" uniformization
of an (ordinary) Abelian variety we are up to (cf. the discussion
after Example~\ref{example:halphen-pen}). Elaborating further the
heuristics of Remark~\ref{remark:heur-strat}, let us collect more
evidence for such kind of a uniformization, thus making a road map
for the upcoming constructions. The strategy is like this: one
first ``localizes" $A$ in some way (to $\widehat{A}$ as above, or
to the \emph{Barsotti\,--\,Tate group} $A[p^{\infty}]$, say);
after that the ``rigidity" of $A$ comes into play (like the
Frobenius action above) to help one handle the patching data for
various ``localizations" of $A$.
\end{remark}

\refstepcounter{equation}
\subsection{Elliptic modules}
\label{subsection:misc-4}

We will now stick to the case of a \emph{global} ground field
$\fie$ having $p = \text{char}~\fie > 0$. Thus, $\fie$ is a finite
extension of the field $\f_p(t)$, where $t$ is a transcendental
parameter.

We are going to introduce another enhancement of formal groups
(cf. Remark~\ref{remark:formal-group-mult-by-p}), called
\emph{formal module} (structure), as well as its refinement,
called \emph{elliptic} (or \emph{Drinfel'd})
\emph{module}.\footnote{Our interest here is merely a consumer's
one (for simply to establish the canonicity property of the
isomorphism $\Theta$ in {\ref{subsection:misc-5}} below, using
Proposition~\ref{theorem:prop-ell-mod-1} and
Lemma~\ref{theorem:prop-ell-mod-2}) and the motivated reader
should turn to more fruitful expositions such as
\cite{deligne-husemoler}, \cite{lafforgue} or
\cite{mumford-kort-de-vries} for instance.} We will be mostly
following the papers \cite{drinfeld-ell-1}, \cite{drinfeld-ell-2}.

Let $G\in\fie[[x]]$ be a formal group as above. All the previous
definitions$\slash$constructions for $G$ may be equally carried
over some coefficient ring $B$ in place of $\fie$ (but still
$\text{char}~B = p$). Now, given $f \in \text{End}(G)\subset
B[[x]]$ we put $D(f) := f'(0)$, thus defining a ring homomorphism
(\emph{differential}) $D: \text{End}(G) \to B$. Assume in addition
that $B$ is an $R$-algebra with a ring homomorphism $e:R\to B$ for
the ring of integers $R\subset\fie$.

\begin{definition}[{cf. \cite[\S 1]{drinfeld-ell-1}}]
\label{definition:form-mod} One calls the pair $(G, \varphi)$ a
\emph{formal $R$-module over $B$} for $\varphi: R \to
\text{End}(G)$ being a ring homomorphism such that $D \circ
\varphi = e$. (The same definition goes over for $\fie$ in place
of $R$. Also, when speaking about formal $R$-modules over a field
$K\supseteq\fie$, we will be silently assuming that $e = $ the
embedding $R\subset\fie$.)
\end{definition}

\begin{example}
\label{example:add-formal-mod} If $G$ is the additive group of
$B$, $G(x, y) = x + y$ (cf. Example~\ref{example:mult-form-gr}),
and $\varphi(f) := fx$ for all $f \in R$, then one gets the
\emph{additive formal $R$-module} $\frak{A}$ (over $B$). The
crucial property of $\frak{A}$ is that for any formal $R$-module
$\frak{B}$ over $B$ there exists an isomorphism
$\frak{B}\stackrel{\iota}{\simeq}\frak{A}$ of formal
$\fie$-modules. Moreover, if one restricts to the case of
$D(\iota) = 1$, then such $\iota$ is unique (see \cite[Proposition
1.2]{drinfeld-ell-1}).
\end{example}

For the additive (algebraic) group $G$ of $B$, one has
$\text{End}(G) = B\{\{\text{Fr}_{p}\}\}$, the ring of all formal
power series $\summ_{i=0}^{\infty}b_i\cdot\text{Fr}_{p}^i$, $b_i
\in B$, where $b_0\cdot\text{Fr}_{p}^0 :=
\left[\text{multiplication by} \ b_0\right]$, so that
$\text{Fr}_{p}\cdot b = b^{p}\cdot\text{Fr}_{p}$ for all $b\in
B$.\footnote{When the (maximal) constant subfield of $\fie$ is
$\f_q$, with $q := p^k$ for some power $k\ge 1$, one should rather
take $\text{Fr}_{q} := \text{Fr}_{p}^k$ (and replace $\f_p$ by
$\f_q$).} Define the embedding $\varepsilon: B \hookrightarrow
B\{\{\text{Fr}_{p}\}\}$ via $b \mapsto b\cdot\text{Fr}_{p}^0$. Let
also $D: B\{\{\text{Fr}_{p}\}\} \to B$ be the differential
homomorphism as above (with
$D\left(\summ_{i=0}^{\infty}b_i\cdot\text{Fr}_{p}^i\right) :=
b_0$). Finally, we introduce the subring
$B\{\text{Fr}_{p}\}\subset B\{\{\text{Fr}_{p}\}\}$ of all
``polynomials" $\summ_{i=0}^{\infty}b_i\cdot\text{Fr}_{p}^i$,
i.\,e. $b_i = 0$ for $i \gg 1$.

From now on $B$ is a field.

\begin{definition}[{cf. \cite[\S 2]{drinfeld-ell-1}}]
\label{definition:ell-mod} A formal $R$-module $(G, \varphi)$ is
called \emph{elliptic} if $\varphi(R) \subseteq
B\{\text{Fr}_{p}\}$ and $\varphi(R)\not\subset\varepsilon(B)$.
\end{definition}

Elliptic modules come along with their \emph{rank}
$d\in\mathbb{N}$ defined by $p^{\deg\varphi(f)} = |f|^d$, $f \in
R$, for the absolute value $|\cdot|$ w.\,r.\,t. to the point
$\infty$ of $\fie$ (see \cite[Propositions 2.1,
2.2]{drinfeld-ell-1}). Further, let $\fie_{\infty}^s$ be the
separable closure of the completion $\fie_{\infty}$ of the field
$\fie$ at $\infty$, and set $B := \fie_{\infty}^s$ in what
follows. Consider $\Lambda \subset B$, a \emph{(period) lattice of
dimension $d$}, i.\,e. $\Lambda$ is a discrete
$\gal\left(\fie_{\infty}^s\slash\fie_{\infty}\right)$-invariant
$R$-module with $d$ generators. Then we define
$$
u(x) := x\prod_{0\ne \alpha\in\Lambda}(1 - \frac{x}{\alpha}).
$$
This is an entire function on $\mathbb{A}^1$ such that $u(x_1 +
x_2) = u(x_1) + u(x_2)$.\footnote{Here the affine line
$\mathbb{A}^1$ is identified with the topological space
$\mathbb{A}^1(\fie_{\infty}^s)$ carrying the rigid topology
induced from $\fie_{\infty}$.}

The essential feature of elliptic modules is the next fundamental

\begin{prop}[{see \cite[\S 3]{drinfeld-ell-1}}]
\label{theorem:prop-ell-mod-1} Let $(G, \varphi)$ be an elliptic
module over $\fie_{\infty}^s$ of rank $d$. Then, in the notations
of Example~\ref{example:add-formal-mod}, there exist $\Lambda =:
\Lambda_{(G, \varphi)}$ and $u =: u_{(G, \varphi)}$ as above such
that $\iota = u$ (for $(G, \varphi) = \frak{B}$). Furthermore, if
$f: G_1 \to G_2$ is a homomorphism of elliptic modules $(G_i,
\varphi_i)$ (i.\,e. $f\in B[x]$ and satisfies $\varphi_2\circ f =
f \circ\varphi_1$), with the associated lattices $\Lambda^i =:
\Lambda^i_{(G_i, \varphi_i)}$,\footnote{Note that $\Lambda^i$ are
necessarily of the same dimension.} then $w\Lambda^1 \subset
\Lambda^2$ for some $w \in \fie_{\infty}^s$. In particular, the
correspondence $(G, \varphi) \rightsquigarrow \Lambda$ is an
equivalence of categories.
\end{prop}

\begin{example}
\label{example:g-m-is-an-ell-mod} With the formal group $G_m$ from
Example~\ref{example:mult-form-gr} we will (canonically) associate
an elliptic module over $\fie_{\infty}^s$. This together with
Theorem~\ref{theorem:manin-wide-a-structure} is the glimpse of a
uniformization we are looking for (cf.
Remarks~\ref{remark:heur-strat},
\ref{remark:ordin-ab-vars-uniform}). Take $f := tx +
x^p\in\text{End}(G_a)$ and consider the group $\cel_{f}$ of all
$x\in\fie_{\infty}^s$ such that $f^k(x) = 0$ for some (variable)
$k\ge 1$ (here $f^k := f\circ\underbrace{\ldots}_{k\
\text{times}}\circ f$ is the $k^{\text{th}}$ iteration of the
endomorphism $f$). Then for the isomorphism $\phi: G_m\to G_a$ the
lattice $\Lambda_m := \phi^{-1}(\cel_{f})$ determines an elliptic
module $(G, \varphi)$ with $\text{End}(G)\simeq\text{End}(G_m)$
canonically. Indeed, one can see that $\varphi(t)(x) = \phi\circ
f\circ\phi^{-1} = \log(1 + e^f - 1) = f$ (hence $\phi$ is an
entire function by Proposition~\ref{theorem:prop-ell-mod-1}) and
$G_m(\bar{x},\bar{y}) = u_{(G, \varphi)}(x+y)$ as (additive) group
laws, where $\bar{x},\bar{y}$ are the classes of $x,y$ modulo
$\Lambda_m$. One may thus think of $G_m$ as being the ``elliptic
curve" $G_a\slash_{\displaystyle\Lambda_m}$.
\end{example}

\begin{remark}
\label{remark:p-groups} In light of the above discussion, any
formal group $G$ can also be identified with its (profinite)
\emph{Dieudonn\'e module} (see \cite[Ch.\,I, \S
4]{manin-fg-big-paper} for instance), which yields an equivalence
between the category of formal groups and discrete modules (over
the ring of Witt vectors enhanced with $\text{Fr}_p$ and
$p\cdot\text{Id}_G$), similar to $\rightsquigarrow$ in
Proposition~\ref{theorem:prop-ell-mod-1}. However, as
Example~\ref{example:g-m-is-an-ell-mod} suggests (see also the
discussion in {\ref{subsection:misc-5}} below), the equivalence
$\rightsquigarrow$ is more fruitful.
\end{remark}

We conclude this subsection by proving the following:

\begin{lemma}
\label{theorem:prop-ell-mod-2} Let $(G, \varphi)$ be an elliptic
module and $f \in\aut~(G)\subset \text{End}~(G)$ its (formal)
automorphism (i.\,e. $f \circ \varphi \circ f^{-1} = \varphi$ and
$D(f) = 1$). Then $f = \text{Id}_G$.
\end{lemma}

\begin{proof}
We may regard $G$ as a formal $\fie$-module $\frak{B}$ (see
Example~\ref{example:add-formal-mod}). Then we have $(f \circ
\iota) \circ \varphi \circ (f \circ \iota)^{-1}(r) =
\iota^{-1}\circ\varphi\circ\iota(r) = [\text{multiplication by
r}]$ for all $r \in R$. Since $D(\iota) = D(f) = D(f \circ \iota)
= 1$, by uniqueness we get $f \circ \iota = \iota$, which implies
that $f = \text{Id}_G$.
\end{proof}

\refstepcounter{equation}
\subsection{Uniformization}
\label{subsection:misc-5}

We keep on with the previous notations. Let us assume in addition
that $A^d$ is ordinary and defined over $\fie$ (as in
{\ref{subsection:misc-4}}). Also, since the objects of primary
interest to us are rully $A^d$'s, especially those equal to
$\prod_{i=1}^d E_i$ for some elliptic curves $E_i\slash\fie$ (cf.
{\ref{subsection:int-66}}), we will assume w.\,l.\,o.\,g. that
$A^d$ is a mod $p~$ reduction of some Abelian
variety$\slash\overline{\ra}(t)$.

Fix some local analytic coordinates $x_1,\ldots,x_d$ on $A^d$ near
$0$ and consider an isomorphism of formal schemes
$$
\widehat{A^d}
\stackrel{\displaystyle\stackrel{\Theta}{\sim}}{\displaystyle\map}
\text{Spec}~\fie[[x_1,\ldots,x_d]].
$$
Replacing $\widehat{A^d}$ by $\widehat{A^d}^{(p^k)}$ for some
$k\ge 1$, we may additionally assume that there is an isogeny
$\widehat{A^d}\map G_m^{\oplus d}$, which factors through $\Theta$
(see Theorem~\ref{theorem:manin-wide-a-structure}). Moreover,
replacing each summand in $G_m^{\oplus d}$ by its Frobenius twist
(see Definition~\ref{definition:formal-group-isog} and
Example~\ref{example:frob}), we reduce to the case of $\Theta$ as
in Theorem~\ref{theorem:manin-wide-a-structure}. Thus $\Theta$ is
an isomorphism of formal groups $G := \widehat{A^d}$ and
$G_m^{\oplus d}$.

For another choice of local analytic coordinates $x'_i$ on $A^d$,
say $x' = \tau(x)$ for some $\tau\in \fie[[x]]$, so we get a
formal group $G'$ together with an isomorphism $\Theta'$, defined
similarly as $\Theta$. In order to extend $\Theta$ onto the entire
$A^d$ (or, if one likes, to uniformize $A^d$), it suffices to show
that $\Theta = \Theta'$ (cf. the proof of
Corollary~\ref{theorem:theta-can} below).

First we need the following:

\begin{lemma}
\label{theorem:for-gr-isom-g} $\tau$ induces an isomorphism $G
\simeq G'$ of formal groups.
\end{lemma}

\begin{proof}
Note that $\widehat{A^d} = \displaystyle\varinjlim_n
\widehat{A^d}/n$ for $\widehat{A^d}/n := \widehat{A^d}
\otimes_{\fie}\text{Spec}\left(\mathcal{O}_{0,
A^d}\slash\frak{m}^n\right)$ and the maximal ideal $\frak{m}
\subset \mathcal{O}_{0, A^d}$. Furthermore, the components of
$\tau(x)$ are all locally convergent, by the choice of $x'$. So it
is enough to prove the assertion for some fixed $\widehat{A^d}/n$
(defined over an Artin ring). But the latter is evident because
the group law on $\widehat{A^d}/n$ is induced by the composition
$A^d \times A^d \longrightarrow A^d$.
\end{proof}

From Lemma~\ref{theorem:for-gr-isom-g} we obtain a commutative
diagram
\begin{eqnarray}
\nonumber G & \stackrel{\tau}{\simeq} & G'\\
\nonumber\downarrow_{\Theta} && \downarrow_{\Theta'}\\
\nonumber G_m^{\oplus d} & = & G_m^{\oplus d}.
\end{eqnarray}
Here the equality $G_m^{\oplus d} = G_m^{\oplus d}$ reads as
$\Theta^{-1*}(x_i) = \Theta'^{-1*}(x'_{\sigma(i)})$ on
$G_m^{\oplus d}$ (for all $i$ and some fixed $\sigma \in S_d$) due
to the next

\begin{prop}
\label{theorem:for-gr-i}  $\Theta'\circ \tau\circ\Theta^{-1} =
\text{Id}_{G_m^{\oplus d}}$ (up to permutations of summands in
$G_m^{\oplus d}$).
\end{prop}

\begin{proof}
The case $d=1$ follows from Lemma~\ref{theorem:prop-ell-mod-2}.
Indeed, the elliptic module structure $\varphi$ associated with
$G_m$ is canonically determined by the group law on $G_m$ (see
Example~\ref{example:g-m-is-an-ell-mod}), which implies that
$f\circ\varphi\circ f^{-1} = \varphi$ for every automorphism
$f\in\aut~(G_m)$. On the other hand, since $\tau$ preserves the
complex volume form on $A^d$, we have $D(f) = 1$ for $f =
\Theta'\circ \tau\circ\Theta^{-1}$.

Now let $d=2$ (the case of arbitrary $d\ge 3$ differs only by more
involved notations). Given $f = \Theta'\circ
\tau\circ\Theta^{-1}$, we write $f =
(f_1,f_2),~f_i\in\fie[[x_1,x_2]]$, and denote by $G_{m,i}$ the
$i^{\text{th}}$ summand of $G_m^{\oplus d}$. After restricting $f$
to $G_{m,1}$, projecting $\text{Im}f$ onto $G_{m,1}$ (via
$(x_1,x_2)\mapsto x_1$) and applying
Lemma~\ref{theorem:prop-ell-mod-2} as above, we obtain that either
$f_1 = 0$ or $x_1$, both modulo the ideal $(x_2)$.\footnote{Up to
interchanges $x_i \leftrightarrow x_j$.} Similarly we have $f_2 =
x_2$ or $0$ mod $(x_1)$. Thus, we may assume that
$f_i=x_i\tilde{f}_i$ for some
$\tilde{f}_i\in\fie[[x_1,x_2]],~\tilde{f}_i = 1$ mod $(x_1,x_2)$.

Further, letting $x_2 := \lambda x_1$ for an arbitrary
$\lambda\in\fie$ and applying Lemma~\ref{theorem:prop-ell-mod-2}
once again, we obtain that both of the automorphisms of $G_{m,1}$,
given by
$$
x_1\mapsto (x_1\tilde{f}_1(x_1,\lambda x_1),~ \lambda
x_1\tilde{f}_2(x_1,\lambda x_1))\stackrel{\text{pr}}{\mapsto} x_1\
(\text{resp.}\ x_2),
$$
are trivial. This implies that $\tilde{f}_i=1$ and completes the
proof of Proposition~\ref{theorem:for-gr-i}.
\end{proof}

Further, we endow variety $A^d$ with a \emph{rigid analytic}
structure (over $\fie_{\infty}$), as defined in \cite{bosch-l} or
\cite{ber} for instance. The preceding discussion condensates to
the next

\begin{cor}
\label{theorem:theta-can} The components of the map $\Theta$ are
entire analytic functions on $A^d$.
\end{cor}

\begin{proof}
Proposition~\ref{theorem:for-gr-i} implies that $\Theta' =
\Theta\circ\tau^{-1}$ is an analytic continuation of $\Theta$.
More precisely, $(\Theta^{-1}\circ\tau^{-1})^*(x_i)$ are (formal)
analytic functions of $x$, for all $i$. Thus, it suffices to show
the components $\Theta_i(x)$ of $\Theta$ are entire functions of
$x$, i.\,e. $\Theta_i(x)$ converges on an analytic $x$-chart
$U_x\subset A^d,1\le i\le d$. We will treat only the case $d = 1$
(the case $d\ge 2$ is similar).

Pick an $n$-torsion point $P\in U_x\setminus{\{0\}}$ for some
$n\ge 2$. Then we have $n\Theta_i(P)\in\Lambda_m$ (see
Example~\ref{example:g-m-is-an-ell-mod}) by construction of
$\Theta$. This shows that $\Theta_i(x)$ is convergent on $U_x$.
\end{proof}

\begin{inter}
In the preceding considerations, we have used essentially two
kinds of arguments, very much typical to the (rigid)
uniformization theories. The first one makes it possible to treat
only the formal situation (i.\,e. that of formal schemes etc.),
employing the rigid topology of a complete ground field $\fie$
when needed (see \cite{mumford-curves-uni},
\cite{mumford-ab-vars-uni}, \cite{mccabe}, \cite{morikawa},
\cite{faltings-chai}, \cite{shaf-shap}). A model example is the
\emph{Tate curve} $E_q := \fie^*\slash
q^{\cel},~q\in\fie,~0<|q|<1$ (see \cite[\S
6]{tate-ellipt-curves}), which is an elliptic curve with the
$j$-invariant equal to
$\displaystyle\frac{1}{q}+744+196884q+\ldots$ (this example had
gotten further development in \cite{mumford-curves-uni},
\cite{mumford-ab-vars-uni} (cf. Remark~\ref{remark:manin-arakelov}
above), where the rigid uniformization was established for the
\emph{totally degenerate} curves and Abelian varieties). The
second type of arguments, dealing with $\text{char}~\fie
> 0$ and bringing a finer analysis into play, uses the ``inner
symmetry" of $\fie$ (as compared to the complex conjugation on
$\fie = \com$), that is the Frobenius $x\mapsto x^p,~x\in\fie$
(cf. Example~\ref{example:frob} and
Remarks~\ref{remark:heur-strat},
\ref{remark:ordin-ab-vars-uniform}). Classical illustration in
this case is the \emph{$p$-adic} (or \emph{Serre\,--\,Tate} if one
prefers) uniformization found in \cite{drinfeld-p-adic} (compare
with constructions in \cite{drinfeld-ell-1}, \cite{katz},
\cite{chai}). Namely, there is a formal scheme $\hat{\Omega}^d$,
parameterizing \emph{special} formal $O_D$-modules over various
(nilpotent) $R$-algebras (cf. {\ref{subsection:misc-4}}). Here $D$
is a central division algebra over $\fie$ with invariant $1/d$ and
$O_D\subset D$ is the ring of integers. Universal formal
$O_D$-module over $\hat{\Omega}^d$ admits a (canonical) system of
finite subgroups $\Gamma_k,k\ge 1$, flat over $\hat{\Omega}^d$,
which are the kernels of the multiplication by $t^k\in R$ (cf.
Example~\ref{example:g-m-is-an-ell-mod}). Now, taking $D$ to be
the quaternion ring over $\ra$, with maximal order $O_D$, the
schemes $\Gamma_k$ can be considered (up to compact factors) as
(formal) modular curves, the latter parameterizing \emph{special}
Abelian $2$-dimensional $O_D$-schemes (\emph{Cherednik's
theorem}).
\end{inter}

At this point, let us return to our model case $d = 1$, $A^d = E$
being an elliptic curve. This is done just to simplify the
notations and the reader should consult the discussion of the
general ($d\ge 2$) case at the end of this section. Further, we
identify $E$ with its ordinary mod $p$ reduction defined over
$\overline{\f}_p$ (cf. {\ref{subsection:misc-3}}), thus regarding
the surface $S := E\times\p^1$ as the curve $E$ over the field
$\fie$ (the one at the beginning of {\ref{subsection:misc-4}}).

Fix another prime $\ell\ne p$ and consider the $\ell$-primary
component $G^{\ell}_S\subset\gal(S) =:
\gal(\overline{\f_p(S)}\slash\f_p(S))$. Using ``topological"
Corollary~\ref{theorem:theta-can} (applied to $E\slash\fie$) we
are going to describe $G^{\ell}_S$ (more or less) explicitly (see
Proposition~\ref{theorem:g-ell-comp} below).

\refstepcounter{equation}
\subsection{Structure of $G^{\ell}_S$}
\label{subsection:misc-6}

Recall that we have constructed a one-to-one correspondence
between the points from $E(\fie_{\infty}^s)$ and some compact
subset in $\mathbb{A}^1(\fie_{\infty}^s)$. For $(x,y)\in E$, the
correspondence in question is given by an entire function
$\Theta(x,y)$ on $E(\fie_{\infty}^s)$, which yields a
\emph{homeomorphism} $\Theta:
E(\fie_{\infty}^s)\stackrel{\sim}{\to}\p^1(\fie_{\infty}^s)$.

Any Galois $\ell$-cover $c:\widetilde{E}\map E$ (over $\fie^s$) is
determined by its ramification divisor $o_1+\ldots+o_l$ and the
monodromy group $G_c$. Also, given a $\fie$-variety $X$, we have
$\gal(X) = \displaystyle\varinjlim_U\pi_1^{\text{alg}}(U)$ for
various Zariski open subsets $U\subset X$. This results in an
exact sequence
$$
1\to \gal(X\otimes\fie^s)\to\gal(X)\to\gal(\fie^s\slash\fie)\to 1
$$
(see e.\,g. \cite[\S 8]{grothendieck-formal-alg-geom}) and a
\emph{bijection}
$$\hbar: G^{\ell}_S\to G^{\ell}_{\p^2}$$ between the Galois
$\ell$-covers of $E\slash\fie$ ($\approx S$) and $\p^1\slash\fie$
($\approx\p^2$). Indeed, given $c:\widetilde{E}\map E$ as above
and the fact that $G^{\ell}_{\p^1}$ is pro-$\ell$-free (see
\cite{pop-gal-cov-curves}), one constructs the Galois cover
$\hbar(c): \widetilde{\p^1}\map\p^1$ (over $\fie^s$) as follows.
First, letting $E(\fie_{\infty}^s) =
\Theta(\p^1(\fie_{\infty}^s))$, we identify $c$ with a
\emph{topological} covering of $\p^1$. Hence $\p^1$ is
(topologically) a $G_c$-quotient of $\widetilde{E}$. Then we
choose $\hbar(c)$ to have the Galois group equal $G_c$ and
coincide with $c$ over an open subset of $\p^1$.\footnote{But the
ramification divisor of $\hbar(c)$ need not coincide with
$\Theta(o_1)+\ldots+\Theta(o_l)$.} Finally, using the presentation
$\p^1 = \widetilde{E}\slash G_c$, one can see the definition of
$\hbar$ is correct.

The next assertion is straightforward (its proof is left to the
reader as an exercise):

\begin{prop}
\label{theorem:g-ell-comp} $\hbar$ is a group
homomorphism.
\end{prop}

Thus we obtain an (outer) isomorphism $G^{\ell}_S\simeq
G^{\ell}_{\p^2}$ of groups. This creates the grounds for applying
results from \emph{anabelian birational geometry} (see
{\ref{subsection:misc-7}} and \cite{bog-tch-rec-funct-fie} for
relevant topics).

\begin{inter}
The previously constructed $\Theta$ points out (again) the
``pathological" nature of the $\text{char}>0$ world (cf.
\cite{esna-hog}, \cite{mumford-patholg-1},
\cite{mumford-patholg-2}, \cite{mehta-srin}, \cite{tate-genus}).
Indeed, $\Theta$ provides an ``analytic" isomorphism between $E$
and $\p^1$, though it does not induce an isomorphism of the
underlying rigid analytic structures -- otherwise $E\simeq\p^1$ as
algebraic curves by the (rigid) GAGA principle.\footnote{This
principle (almost tautologically) leads to an isomorphism
$G^{\ell}_S\simeq G^{\ell}_{\p^2}$ when $E = \fie^*\slash
q^{\cel}$ is the Tate curve. So
Proposition~\ref{theorem:g-ell-comp} (as well as its proof) may be
considered as a generalization of this fact (without direct appeal
to the rigid GAGA).} In particular, the function $\Theta$ does not
belong to the \emph{Tate algebra} of the local ring
$\mathcal{O}_{0,E}$, i.\,e. if $x\in\mathcal{O}_{0,E}$ is a local
parameter and $\Theta = \sum a_ix^i,~a_i\in\fie_{\infty}^s$, then
$|a_i|\not\to 0$ as $i\to\infty$. However, in line with Intermedia
at the end of {\ref{subsection:misc-5}}, our construction of the
above isomorphism $\hbar$ is (a sort of) a reminiscent of the
approach developed in \cite{pop-gal-cov-curves} to prove the
(geometric) Shafarevich's (resp. Abhyankar's) Conjecture. Indeed,
in \cite{pop-gal-cov-curves} one proceeds from local to global,
starting with a unit (rigid) disk $D$ (a \emph{germ of local
behavior}) in a given affine curve $U\slash{\bar{\fie}}$, after
what one constructs a Galois covering of $D$ (\emph{rigid Galois
cover}) and glues (analytically) all this data together, for
various $D\subset U$, to obtain a prescribed element in
$\pi_1^{\text{alg}}(U)$. Canonicity (aka correctness) of this
construction is ensured by the existence of Galois cyclic covers
of $\p^1\setminus{\{0,\infty\}}$, ramified only at $0$ and some
other $c\in\bar{\fie}$, with $|c|<1$. The latter, in turn, follows
simply from the Kummer and Artin\,--\,Schreier theories -- in
contrast with the ``analytic" arguments of
{\ref{subsection:misc-5}}.
\end{inter}

\refstepcounter{equation}
\subsection{Concluding discussion}
\label{subsection:misc-7}

We now return to the ground field $\fie = \overline{\f}_p$.

The isomorphism $G^{\ell}_S\stackrel{\hbar}{\simeq}
G^{\ell}_{\p^2}$ and the main theorem of \cite{pop-pro-ell-I}
yield a field $F$ together with (finite) Galois $\ell$-extensions
$F\supseteq\fie(S),~F\supseteq\fie(\p^2)$. More precisely,
according to \cite{pop-pro-ell-I} both $\fie(S)$ and $\fie(\p^2)$
have common purely inseparable closure, $F^{\text{pi}}$ say. Then,
using the canonical bijection
$\text{Out}\big(G^{\ell}_S,G^{\ell}_{\p^2}\big)\leftrightarrow\aut(F^{\text{pi}})$
between the outer isomorphisms of Galois groups and the field
automorphisms$\slash\fie$ (up to Frobenius twists), we pick
$G^{\ell}_S\simeq G^{\ell}_{\p^2}$ corresponding to
$\text{id}\in\aut(F^{\text{pi}})$, thus finding our $F$. This
results in a diagram
$$
\xymatrix{
&&W\ar@{->}[ld]_{r}\ar@{->}[rd]^{r'}&&\\
&S&&\p^2,&}
$$
where both $r,r'$ are some Galois $\ell$-covers (with isomorphic
Galois groups acting birationally on $W$), superseded probably by
p.\,i. maps. Furthermore, since there is no a priori preference in
the choice of the extension $F\supseteq\fie(S)$, we pick the
latter to be such that $r$ is unramified over generic fiber of the
projection $S\map\p^1$ (we are neglecting the Frobenius action for
the sheer transparency).

Fix some (generic) $t_0\in\fie$ and consider $E_{t_0} = E \times
t_0\subset S$. Let $E^r_{t_0}$ be the normalization of the scheme
preimage $r^{-1}(E_{t_0})$ in $\fie(W)$.

\begin{lemma}
\label{theorem:e-r-is-frob-e} $E^r_{t_0}$ is an elliptic curve.
\end{lemma}

\begin{proof}
Suppose first that $r$ is p.\,i. of degree $p^k$ for some $k\ge
1$. Let us assume for simplicity that $k=1$. Then the extension
$\fie(W)\supseteq\fie(S)$ is given by an equation $z^p =
f(x,y,t)$. Here $t$ is a coordinate on $\p^1$, $x,y$ are
coordinates on the affine model of $E$ (cf.
{\ref{subsection:misc-1}}), $z$ is an extra variable and
$f\in\fie(x,y,t)$. Replacing $E = E_{t_0}$ by $E^{(p)}$ (cf.
Example~\ref{example:frob}), we get a morphism $r_1: W\map
E^{(p)}\times\p^1 =: S_1$, which is a prolongation of $r$ via the
natural (Frobenius) map $S\map S_1$. Now, extension
$\fie(W)\supseteq\fie(S_1)$ is given by the equation $z^p =
f(x^p,y^p,t)$, which implies the normalization of
$r_1^{-1}(E^{(p)})$ ($= E^r_{t_0}$) is birational to $E^{(p)}$.
Indeed, $r_1^{-1}(E^{(p)})$ coincides with the locus $z^p =
f(x^p,y^p,t_0)$, so that its normalization is $\approx\
[\text{graph of the function}\ f(x,y,t_0) \ \text{on}\ E]$.

We go on to the case of arbitrary $r$. One may assume
w.\,l.\,o.\,g. that $r$ is separable. Then, by construction, the
covering $r: E^r_{t_0}\map E$ is unramified, thus $E^r_{t_0}$ is
an elliptic curve.
\end{proof}

Replace $E_{t_0}$ by a Frobenius twist of $E^r_{t_0}$ if necessary
to obtain the morphism $r'\big\vert_E$ (cf.
Lemma~\ref{theorem:e-r-is-frob-e}) is smooth. Again, assuming
w.\,l.\,o.\,g. $r'$ to be separable, we find that $r'(E) = E\slash
G$ for the Galois group $G$ of the covering $r'$. Finally, since
$G$ is an $\ell$-group, with $\ell\gg 1$ say, the quotient
$E\slash G$ is an elliptic curve. This shows (again) that $\p^2$
is b.-\,H. (over $\overline{\f}_p$).

\begin{inter}
One may spot a similarity between the preceding considerations and
those of paper \cite{bog-korot-tsch}. Recall that for any two
smooth curves $C,\bar{C}$ over $\fie=\overline{\f}_p,~p\ge 3$, an
\emph{isomorphism of pairs} $(C,J)\to (\bar{C},\bar{J})$ implies
the corresponding Jacobians $J$ and $\bar{J}$ are isogeneous (see
\cite[Theorem 1.2]{bog-korot-tsch}). In particular, given a curve
$C$ of genus $>4$, the maximal abelian quotient
$\mathcal{G}^{a}_K$ of the prime-to-$p$ part of the Galois group
$\gal(\overline{K}/K),~K := \fie(C)$, together with the collection
$\mathcal{I} := \{I_{\nu}\}_{\nu}$ of inertia subgroups
$I_{\nu}\subset\mathcal{G}^{a}_K$, determines the isogeny class of
the Jacobian of $C$ (see \cite[Theorem 9.3]{bog-korot-tsch}). We
recall that in the latter setting the isomorphism of pairs $\phi:
(C,J)\to (\bar{C},\bar{J})$ implies that two natural actions of
the Frobenius $\text{Fr}^k_p,~k\gg 1$ (when both considered in
$\text{End}(C)$), on the Tate modules
$T_{\ell}(J),T_{\ell}(\bar{J})$ commute. The idea then is to
reduce to the case $k=1$ and reconstruct an isogeny $J\map\bar{J}$
via \cite{tate-isog-frob}. Now, given two isomorphic data
$(\mathcal{G}^{a}_K,\mathcal{I})$ and
$(\mathcal{G}^{a}_{\bar{K}},\bar{\mathcal{I}})$ for $C$ and
$\bar{C}$, respectively, one essentially recovers an isomorphism
of pairs $\phi$ from isomorphic $\text{Div}^0$'s.
\end{inter}

\begin{remark}
\label{remark:further-path} Note that the morphism $r'$ can not be
purely inseparable. Indeed, otherwise the surface $W$ (hence also
$S$) is rationally connected, which forces $S$ to be rational (for
$S$ is \emph{ruled}). This shows in particular that the condition
$\text{Pic}(X) = \text{NS}(X)$ in \cite[Theorem
1]{fed-yu-univ-spaces-2} \emph{is} actually necessary for the
claim to be true as stated (compare with
\cite{fed-yu-univ-spaces-3}). Let us also remark that $W$ may be
considered as a $2$-dimensional analog of the \emph{universal
curve} found in \cite{fed-yur-unram-cor}. This is another
(supporting) side of the ``ideology" adopted in the present
section in order to relate the function fields $\fie(\p^2)$ and
$\fie(S)$.
\end{remark}

Up to this end all the arguments apply literarilly to
$\prod_{i=1}^nE_i$ in place of $E$ (respectively, to
$\prod_{i=1}^nE_i\times\p^1$ in place of $S$, $\p^{n+1}$ in place
of $\p^2$, etc.). Hence Conjecture~\ref{theorem:p-n-ab-pen} would
be proven once we showed the birational embedding $r':
\prod_{i=1}^nE_i\times t_0\dashrightarrow\p^{n+1}$, a priori
defined over $\overline{\f}_p$, can be lifted to $\text{char}~0$.

Note at this point that $r'$ corresponds to a $(n+1)$-dimensional
linear subsystem $\subseteq
H^0\big(\prod_{i=1}^nE_i,\mathcal{L}\big)$ for some
$\mathcal{L}\in\text{Pic}\big(\prod_{i=1}^nE_i\times_{\fie}\overline{\f}_p\big)$.
Then, since $r'$ is separable and generically $1$-to-$1$ on
$\prod_{i=1}^nE_i$, it suffices to show that $\mathcal{L}$ admits
a $\text{char}~0$ lifting. The latter is easy when $n = 1$, for
$\text{Pic}(E) = \cel\oplus\text{Pic}^0(E)$, and in the case $n\ge
2$ it suffices to take $E_i$ to be pair-wise non-isogeneous for
all $i$, so that $\text{Pic}\big(\prod_{i=1}^nE_i\big) =
\prod_{i=1}^n\text{Pic}(E_i)$. We leave this as an exercise to the
reader. (Hint: use Example~\ref{example:hasse}.)

\bigskip

\section{Miscellania}
\label{section:misce}

\refstepcounter{equation}
\subsection{}
\label{subsection:misce-1}

We would like to conclude by discussing another, geometric,
variant of the point of view we have tried to advocate throughout
the paper, namely that ``an Abelian variety $A^n$ with many
symmetries is (birationally) close to $\p^n$, after one mods out
$A^n$ by these symmetries." As a model illustration, we took $A^n
:= E_1\times\ldots\times E_n$, with elliptic curves $E_i$ defined
over a global field $\fie$ having $\text{char}~\fie = p>0$, and
acted this $A^n$ by (various powers of) the Frobenius. Then the
upshot of our thesis in this case are the results of subsections
{\ref{subsection:misc-5}}, {\ref{subsection:misc-6}} and
{\ref{subsection:misc-7}}. Geometric counterpart of this ideology
is as follows.

Consider $A^n := E^n$, where $E$ is an elliptic curve over $\fie =
\com$, and identify $A^n$ with $\text{Hom}(\Lambda,E)$ for a full
lattice $\Lambda\subset\mathbb{R}^n$. Thus there is a natural
$GL(n,\cel)$-action on $A^n$. More specifically, we consider
$R\subset GL(n,\cel)$, a reflection group, for which one can form
the quotient $Q := E^n/R$. Then it is expected (in accordance with
the edifice from the last paragraph) that variety $Q$ should be
``close to rational."

\begin{example}
\label{example:ed-looij} Classically, when $\Lambda$ is the
lattice w.\,r.\,t. the (reduced irreducible) root system for a
Weyl group $R$, it was shown in the beautiful paper \cite{looij}
that $Q$ is a weighted projective space. More precisely, let
$\Lambda$ span the linear space $V$, so that $n = \dim V$ is the
rank of the corresponding root system. Let $\alpha\in V$ be the
highest root and $\alpha^{\vee}\in V^*$ be its dual. This
$\alpha^{\vee}$ is a root in the dual root system and we can write
$\alpha^{\vee} = \displaystyle\sum_{i=1}^n g_i\alpha_i$ for the
basic roots $\alpha_i$ and some $g_i\in\mathbb{N}$. Then $Q =
\p(1,g_1,\ldots,g_n)$.
\end{example}

The setting of Example~\ref{example:ed-looij} was generalized
further in \cite{kol-lars}. Namely, it was shown in
\cite{kol-lars} (among other things) that the quotient of an
Abelian variety $A^n$ by a finite group $G\subset\aut(A^n)$ is
\emph{rationally connected} iff the $G$-action on the tangent
space $T_0(A^n)$ is irreducible and violates the
\emph{Reid\,--\,Tai property} (see \cite[Corollary 25]{kol-lars}).
In particular, Weyl groups from Example~\ref{example:ed-looij}
(with $A^n = E^n$) are like this, but at the same time the
following question seems to be out of reach (cf. \cite[Section
5]{kol-lars}):

\renewcommand{\thequestion}{K-L}

\begin{question}
\label{question:kol-lar} Let $A^n = E^n$ and $G$ be as in
Example~\ref{example:ed-looij}, but suppose now that $G$ is
\emph{any} finite group acting (isometrically, say) on $\Lambda$,
with induced $G$-action on $T_0(A^n)$ irreducible and non-RT. Is
then the quotient $X := A^n\slash G$ rational (or at least
unirational)?
\end{question}

Note that positive answer to Question~\ref{question:kol-lar} might
give an alternative way of proving
Conjecture~\ref{theorem:p-n-ab-pen} (for one may choose a general
Abelian hypersurface $A_t\subset A^n$ and proceed as in the proof
of Corollary~\ref{theorem:main-cor-a}). Let us outline how one
could answer Question~\ref{question:kol-lar} for \emph{complex
crystallographic reflection groups}.

Fix the quotient map $q: Y := E^n \map Y\slash G = X$ and write
\begin{equation}
\label{hurw-form} 0 = K_Y \equiv q^*\big(K_X + \sum_i
\frac{r_i-1}{r_i}R_i\big)
\end{equation}
for some $r_i\in\cel_{\ge 0}$ and $R_i\in\text{Weil}(X)$ (Hurwitz
formula). Note that $\#R_i = n+1$ by the assumption on $G$ (see
the list in \cite{popov}).

Consider a $G$-invariant divisor $H :=
q^*(\mathcal{O}_X(1))\in\text{Pic}(Y)$. The Appell\,--\,Humbert
theorem (see \cite{mumfor-ab-vars}) identifies $H$ with a positive
definite Hermitian form $I$ on $\com^n$ such that $\text{Im}(I)$
is skew-symmetric and $\cel$-valued on the lattice
$\Lambda\otimes_{\cel}H^1(E,\cel)$ ($= H^1(Y,\cel)$ canonically).
Note also that $I$ is $G$-invariant by construction. In
particular, since the $G$-representation $\com^n =
H^1(Y,\cel)\otimes\com$ is irreducible, the form $I$ is unique (up
to a constant multiple), which implies that $\text{Pic}(X) = \cel$
and $X$ is a \emph{log Fano} (see \cite{isk-prok}). Our intention
here is to prove that $X$ is \emph{toric}. In order to do this we
employ the following:

\begin{conj}[V.\,V. Shokurov]
\label{theorem:slava} Let $X$ be a normal $\ra$-factorial
algebraic variety with a $\ra$-boundary divisor $D = \sum d_i D_i$
($0\le d_i\le 1$ and $D_i$ are prime Weil divisors) such that

\begin{itemize}

    \item the divisor $-(K_X+D)$ is nef,

    \smallskip

    \item the pair $(X,D)$ is log canonical.

\end{itemize}
Then the estimate $\sum d_i \leq \text{rk}~\text{Pic}(X) + \dim X$
holds. Moreover, the equality is achieved iff the pair
$(X,\llcorner D\lrcorner)$ is toric, i.\,e. $X$ is toric and
$\llcorner D\lrcorner$ is its boundary.
\end{conj}

Recall that Conjecture~\ref{theorem:slava} had been verified in
the case $\text{Pic}(X) = \cel$ (see \cite[Corollary 2.8]{pro-1})
and so we can apply it to our $X = Y\slash G$ (see
\cite{ilya-toric} for further discussion of
Conjecture~\ref{theorem:slava} and its variations). Namely, fix
some generator $\mathcal{O}_X(1)$ of $\text{Pic}(X)$, so that $R_i
= \mathcal{O}_X(r_im_i)$ in \eqref{hurw-form} for some integers
$m_i\ge 1$ and all $i$. Then to show that $X$ is toric it suffices
to establish the pair $(X,\displaystyle\sum_i R_{i,X})$ is log
canonical for generic $R_{i,X}\in |\mathcal{O}_X(r_i-1)|$ (recall
that $\#R_i = n+1$, i.\,e. we handle all the conditions in
Conjecture~\ref{theorem:slava}, except for the lc one). The
latter, in turn, is easily seen to be equivalent to the same
statement about $Y$ in place of $X$ and generic $G$-invariant
divisors $R_{i,Y}\in |H^{r_i-1}|$ in place of $R_{i,X}$,
respectively. Then we can construct $G$-invariant theta
characteristics $\theta_i$ corresponding to $R_{i,Y}$ exactly as
in \cite[\S 4]{looij}, taking $I$ to be the Coxeter matrix
w.\,r.\,t. $G$, say. At this point, one can do without knowing
algebraic relations between $\theta_i$ in the ring
$\bigoplus_{m\ge 0}H^0(Y,H^m)$ (as it was in \cite{looij}), and
all one has to check is that the loci $\theta_i = 0$ (for all $i$
and generic $\theta_i$) determine a log canonical pair.

\begin{remark}
\label{remark:comp-looij} Obviously, the remaining assertion one
is left with (in order to prove that $X$ is toric) is easier to
deal with and applies in the more general setting, compared to the
the proof of Theorem 3.4 in \cite{looij}. However, the preceding
arguments can only show that $X$ is a \emph{fake weighted
projective space} (see \cite{kaspr}), lacking the explicit
description of $X$ in terms of $G$ as in
Example~\ref{example:ed-looij}.
\end{remark}

\refstepcounter{equation}
\subsection{}
\label{subsection:misce-3}

Regrettably, due to the lack of space and the author's competence
we could not touch many other topics, concerning (both birational
and biregular) geometry of Abelian fibrations. We thus restrict
ourselves to simply stating some of these (together with
formulation of relevant problems), hoping the technique (and
results) of the present paper (in interaction with methods of the
articles we are going to mention) will bring further insight and
development to this beautiful subject.

\renewcommand{\theenumi}{\Alph{enumi}}

\begin{enumerate}

\item\label{1-ita}

In view of Conjecture~\ref{theorem:p-n-ab-pen} and constructions
in Section~\ref{section:misc}, it would be interesting to find out
whether an Abelian fibration $f: \p^n\dashrightarrow B$, with
$\dim X \ge \dim B + 2$, is isotrivial, and if it is not, does it
possess a section. Again, comparing with the $n=2$ situation, one
should not expect ``many" (up to birational equivalence)
non-isotrivial $f$'s, while for isotrivial ones there should exist
a complete (effective?) description (cf.
Example~\ref{example:halphen-pen}). Leaving the non-isotrivial
problem (which is of the existence-type and should be the most
difficult) silently aside, for isotrivial $f$ the case of
particular interest is when $\dim B$ is small, say $\dim B \le 2$.
If this holds, one may apply the technique of \cite{libgober} for
example to estimate the Mordell\,--\,Weil group, discriminant
locus, etc. for the desingularized $f$.

\item\label{2-ita} The method of Section~\ref{section:misc}
is implicit and it would be interesting to construct Abelian
fibrations on rational varieties explicitly (as we have attempted
to do in {\ref{subsection:misce-1}}). Classically, there are
Horrocks\,--\,Mumford Abelian surfaces in $\p^4$ (see
\cite{hor-mum}). This result was extended in \cite{hulek} to find
Abelian surfaces in $\p^1\times\p^3$, then in
$\p(\mathcal{O}_{\p^2}\oplus\mathcal{O}_{\p^2}(1)\oplus\mathcal{O}_{\p^2}(1))$)
(see \cite{sank}), etc. In the beautiful papers \cite{gross-pop},
\cite{gross-pop-1}, one finds pencils of Abelian surfaces lying on
Calabi\,--\,Yau $3$-folds, and for a birational construction (of
rational fibrations on $\p^5$ by degree $3$ nodal surfaces) we
refer to \cite{totar}. On this way, one obtains a source of
explicit examples of Abelian fibrations on $\p^n,~n\gg 1$.

\item\label{3-ita} We would like to distinguish/characterize b.-\,H. structures on a rationally connected $X$ ($=\p^n$
for instance) by existence of a certain sheaf $\mathcal{L}$ on $X$
with particular values of
$c_i(\mathcal{L}),~\text{rk}~\mathcal{L}$, etc. Also we would like
to know how b.-\,H. behaves under small deformations. See
\cite{kollar-ellipt-cy} for relevant discussion and results.

\item\label{4-ita} Finally, it might be of some interest (and use)
to compare the matters in (\ref{1-ita}), (\ref{2-ita}),
(\ref{3-ita}) with similar ones for symplectic (or Poisson, or
$\ldots$) manifolds in place of rational varieties (see
\cite{beau-holsymplect}, \cite{sawan} for an overview of related
problems).

\end{enumerate}

\bigskip

\thanks{{\bf Acknowledgments.}
I am grateful to F. Bogomolov for numerous helpful discussions and
advice. Thanks to I. Cheltsov, A. Kuznetsov, Yu.G. Prokhorov, C.
Shramov, and K. Zainoulline for fruitful conversations. Some parts
of the manuscript were prepared during my visit to Stanford
University in July 2012 and I am grateful to Y. Rubinstein for
hospitality. Finally, the present work was completed in a friendly
and stimulating atmosphere of the Max Planck Institute for
Mathematics, which I am happy to acknowledge.

\end{document}